\documentclass[letterpaper, 11pt]{amsart}

\usepackage{amsthm}
\usepackage{amssymb}
\usepackage{amsmath}
\usepackage{amscd}
\usepackage{pgfplots}
\usepackage{sidecap}
\usepackage[margin=1in]{geometry}

\usetikzlibrary{datavisualization}
\usetikzlibrary{datavisualization.formats.functions}
\usetikzlibrary{intersections}

\newtheorem{prop}{Proposition}
\newtheorem{lemma}{Lemma}
\newtheorem{thm}{Theorem}
\newtheorem{cor}{Corollary}

\theoremstyle{definition}
\newtheorem{defn}{Definition}
\newtheorem{ex}{Example}

\DeclareMathOperator{\ord}{ord}
\DeclareMathOperator{\Res}{Res}
\DeclareMathOperator{\ordRes}{ordRes}

\DeclareMathOperator{\MinResLoc}{MinResLoc}
\DeclareMathOperator{\PGL}{PGL}
\DeclareMathOperator{\BDV}{BDV}
\DeclareMathOperator{\CPA}{CPA}
\DeclareMathOperator{\Adj}{Adj}
\DeclareMathOperator{\slv}{slv}
\DeclareMathOperator{\GL}{GL}
\DeclareMathOperator{\Bary}{Bary}
\DeclareMathOperator{\supp}{supp}

\DeclareMathOperator{\diam}{\textrm{diam}}
\newcommand{\pberk}{\textbf{P}^1_{\textrm{K}}}
\newcommand{\hberk}{\textbf{H}^1_{\textrm{K}}}
\newcommand{\aberk}{\textbf{A}^1_{\textrm{K}}}
\newcommand{\zetaG}{\zeta_{\text{G}}}
\newcommand{\hsia}[3]{\delta(#1, #2)_{#3}}
\newcommand{\del}{\partial}
\newcommand{\dlim}{\displaystyle\lim}

\newcommand{\vv}{\vec{v}}
\newcommand{\vw}{\vec{w}}

\newcommand{\B}{\textrm{B}}
\newcommand{\CC}{\mathbb{C}}
\newcommand{\ZZ}{\mathbb{Z}}
\newcommand{\PP}{\mathbb{P}}
\newcommand{\RR}{\mathbb{R}}

\begin{document}
\title{Equidistribution of the Crucial Measures in non-Archimedean Dynamics}

\author{Kenneth Jacobs}
\address{Kenneth Jacobs\\ 
Department of Mathematics\\
University of Georgia\\
Athens, Georgia 30602\\
USA}
\email{kjacobs2@uga.edu}


\begin{abstract}
Let $K$ be a complete, algebraically closed, non-Archimedean valued field, and let $\phi\in K(z)$ with $\deg(\phi) \geq 2$. In this paper we consider the family of functions $\ord\Res_{\phi^n}(x)$, which measure the resultant of $\phi^n$ at points $x$ in $\pberk$, the Berkovich projective line, and show that they converge locally uniformly to the diagonal values of the Arakelov-Green's function $g_{\mu_{\phi}}(x,x)$ attached to the canonical measure of $\phi$. Following this, we are able to prove an equidistribution result for Rumely's crucial measures $\nu_{\phi^n}$, each of which is a probability measure supported at finitely many points whose weights are determined by dynamical properties of $\phi$.

\end{abstract}
\maketitle

\section{Introduction}
Let $(K, \lvert\cdot\rvert)$ be a complete, algebraically closed, non-Archimedean valued field, $\mathcal{O}_K$ its ring of integers, and $\mathfrak{m}_K$ its maximal ideal. Denote by $k$ its residue field $k = \mathcal{O}_K / \mathfrak{m}_K$. We normalize the absolute value on $K$ so that $\log_v(|x|) = -\ord_{\mathfrak{m}_K}(x)$.

This paper is concerned with the dynamics of a rational map $\phi\in K(z)$ of degree $d\geq 2$ on the Berkovich projective line over $K$, which we denote $\pberk$. Rumely has recently introduced two equivariants attached to such a map that carry information about the reduction of conjugates of $\phi$ \cite{Ru1, Ru2}. The first equivariant is a function $\ordRes_\phi:\pberk \to \RR$, which measures the resultant of $\GL_2(K)$-conjugates of a homogeneous lift of $\phi$; if $\phi$ has potential good reduction, the locus where this function is minimized identifies the conjugate realizing good reduction. 

The second equivariant is a probability measure $\nu_\phi$ called the crucial measure, which is defined as a weighted sum of point masses (see \cite{Ru2} Theorem 6.2): $$\nu_{\phi} := \dfrac{1}{d-1} \sum_{P\in \hberk} w_{\phi}(P)\delta_P\ ;$$ here, $w_{\phi}(P)$ is a certain weight function which vanishes on points of type I, III and IV and whose values at type II points are determined by the reduction of $\phi$ at the corresponding point $P\in \hberk$; an explicit formula is given in \cite{Ru2} Definition 8. In particular, the weight function is integer valued, and only finitely many points have positive weight. Rumely showed that the formulation of $\nu_\phi$ given above arises naturally when computing the Laplacian of $\ordRes_\phi$ on a canonical subtree $\Gamma_{\widehat{\textrm{FR}}}$ of $\hberk$ (see \cite{Ru2} Corollary 6.5). 

These equivariants have been used to establish several important facts in non-Archimedean dynamics. Using the measures $\nu_\phi$, Rumely has shown that $\phi$ can have at most $d-1$ repelling type II points in $\pberk$ (see \cite{Ru2} Corollary 6.3); prior to this, it was unknown whether there were always finitely many such points. Rumely has also used these measures to show that semistable reduction of $\phi$ is equivalent to minimality of the resultant (see \cite{Ru2} Theorem 7.4. Szpiro, Tepper and Williams had previously shown that semi-stable reduction {\it implies} minimality of the resultant using different techniques, and their result holds for morphisms on higher dimension projective spaces; see \cite{STW} Theorem 3.3). In \cite{DJR} Doyle, the author, and Rumely used $\ordRes_\phi$ and $\nu_\phi$ to show that for quadratic rational maps, the points in $\supp(\nu_\phi)$ determine the class of the reduction $\widetilde{\phi}$ in the moduli space $\mathcal{M}_2(k)$. \\

In this paper, we consider the corresponding equivariants attached to the iterates of $\phi^n$. Our first result concerns the functions $\ordRes_{\phi^n}$:

\begin{thm}\label{thm:fnconv}
For $x\in \hberk$, the normalized functions $$\frac{1}{d^{2n}-d^n}\ord\Res_{\phi^n}(x)$$ converge to the diagonal values of the Arakelov-Green's function, $g_{\phi}(x,x)$ of $\phi$. The convergence is locally uniform on $\hberk$ in the strong topology. 
\end{thm} 

\noindent We state and prove a more explicit version of Theorem~\ref{thm:fnconv} in Section \ref{sect:fnconv} below (see Theorem~\ref{thm:fnconvexplicit}); in particular, we are able to show that the error term is $O\left(\frac{1}{d^n}\right)$. Our proof comes from a decomposition of $\ordRes_\phi$ into three terms, which closely parallels a decomposition of $g_\phi(x,x)$ given in \cite{BR} Chapter 10 (see also Table~\ref{table:decomp} below).

Our second main result is the equidistribution of the crucial measures $\{\nu_{\phi^n}\}$ attached to the iterates $\phi^n$:

\begin{thm}\label{thm:wkconv}
The measures $\nu_{\phi^n}$ converge weakly to the canonical measure $\mu_{\phi}$.
\end{thm}
\noindent The canonical measure $\mu_{\phi}$ is the unique $\phi$-invariant measure which does not charge the exceptional set and which satisfies $\phi^* \mu_\phi = d\cdot \mu_\phi$ (see \cite{FRL2}).  Here too, an explicit version of the theorem will be given for a dense class of test functions (see Theorem~\ref{thm:wkconvexplicit}). As was the case in Theorem~\ref{thm:fnconv}, the error term is $O\left(\frac{1}{d^n}\right)$. We give explicit examples of the equidistribution for the map $\phi(z) = \frac{z^p-z}{p}$ over $\CC_p$, $p\geq 3$, and for flexible Latt\`es maps (see Section~\ref{sect:exampleswkconv}). 

Our main tool in proving Theorem~\ref{thm:wkconv} is an explicit computation of the Laplacian of $\ordRes_\phi$ on {\it arbitrary} subtrees of $\hberk$, augmenting a result of Rumely (\cite{Ru2} Corollary 6.5) who computed the Laplacian on the subtree $\Gamma_{\textrm{Fix, Repel}}$ spanned by the type I fixed points and type II repelling fixed points. These computations are carried out in Sections~\ref{subsect:slopesagain} and~\ref{sect:applap}.\\

We next apply Theorems~\ref{thm:fnconv} and~\ref{thm:wkconv} to the study of the Minimal Resultant Locus, $\MinResLoc(\phi)$, which is the set of points where $\ordRes_\phi$ attains its minimum. If $\phi$ has potential good reduction, it is a single type II point in $\hberk$ corresponding to a $\PGL_2(K)$-conjugate $\phi^\gamma$ which has good reduction. Otherwise, $\MinResLoc(\phi)$ is either a point or a segment in $\hberk$ (\cite{Ru1} Theorem 0.1). From a measure theoretic perspective, $\MinResLoc(\phi)$ is the barycenter of the crucial measure $\nu_\phi$ (\cite{Ru2} Theorem 7.1; a formal definition of barycenter in this context is given in Section~\ref{sect:barycenters} below).

In Section~\ref{sect:barycenters} below, we show that $\MinResLoc(\phi^n)$ is closely related to the barycenter $\Bary(\mu_\phi)$ of the canonical measure $\mu_\phi$. If $\phi$ has potential good reduction, $\Bary(\mu_\phi)$ is also a single type II point corresponding to a $\PGL_2(K)$-conjugate $\phi^\gamma$ which has good reduction; otherwise, it is a point or a segment in $\hberk$. Moreover, $\Bary(\mu_\phi)$ is the collection of points where the diagonal Arakelov-Green's function $g_\phi(x,x)$ is minimized. These facts about the barycenter are originally due to Rivera-Letelier, but as their proofs have not yet been published we have included our own in Section~\ref{sect:barycenters}.

Our main result concerning the sets $\MinResLoc(\phi^n)$ is the following: 

\begin{thm}\label{thm:main:baryresults}
Suppose $\textrm{char}(K) = 0$. Let $\phi \in K(z)$ have degree $d\geq 2$, and let $R=\frac{2}{d-1}\ord\Res(\phi)$. Let $B_\rho(\zeta, R)$ denote the ball of radius $R$ about $\zeta$ in the $\rho$-metric on $\hberk$.
\begin{itemize}
\item[] $[A]$\hspace{0.25cm}(Approximation) For every $\epsilon>0$, there exists an $N$ so that $\MinResLoc(\phi^n)$ is contained in the $\epsilon$-ball around $\Bary(\mu_\phi)$ for every $n\geq N$. 
\item[] $[B]$\hspace{0.25cm}(Uniform Bounds) For every $n$ we have $\MinResLoc(\phi^n) \subseteq B_\rho(\zetaG, R)$; moreover, there is a constant $M$ depending only on $\phi$ such that $\Bary(\mu_\phi) \subseteq B_\rho(\zetaG, R+M)$.
\end{itemize}
\end{thm}

Part [A] is established in Proposition~\ref{prop:minresepsilon} below; there, we rely on explicit estimates for the slope of the diagonal Arakelov-Green's function $g_\phi(x,x)$. Part [B] is established in Corollary~\ref{cor:boundonminresloc} and Proposition~\ref{prop:RboundsonBary}; in proving these results, we use an explicit estimate for the rate of growth of the constant and leading coefficients of homogeneous lifts $\Phi^n$ of $\phi^n$.

\subsection{Outline of the Paper}

The rest of the paper is divided into six sections. In Section \ref{sect:convandnot}, some conventions and notations concerning the Berkovich projective line $\pberk$ and dynamics on $\pberk$ are developed. In Section \ref{sect:fnconv} we prove Theorem \ref{thm:fnconvexplicit}, which is a more explicit version of Theorem \ref{thm:fnconv}. 

Following this, in Section \ref{sect:lapconv}, we set out to show the weak convergence of the family of crucial measures. For this we develop formulae similar to those given in \cite{Ru2}, Propositions 5.2-5.4 for the slopes of $\ord\Res_{\phi^n}(x)$ on subtrees in $\hberk$. To show weak convergence, we first prove Theorem \ref{thm:wkconvexplicit}, which gives an explicit convergence estimate for test functions that are continuous, piecewise affine functions on fixed finite subtrees $\Gamma\subseteq\hberk$. An approximation theorem for arbitrary continuous functions (see \cite{BR} Proposition 5.4) allows this result to be extended as needed to show weak convergence. 

In Section~\ref{sect:barycenters}, we define the barycenter of a finite, positive Radon measure on $\pberk$ and give some of its basic properties, all of which are originally due to Rivera-Letelier. Using the estimates in Sections~\ref{sect:fnconv}--\ref{sect:lapconv}, we show that the sets $\MinResLoc(\phi^n)$ all lie near to $\Bary(\mu_\phi)$, and we give an explicit example that rules out Hausdorff convergence in general. Finally, in Section~\ref{sect:barybounds} we establish the uniform bounds on $\MinResLoc(\phi^n)$ and $\Bary(\mu_\phi)$. 

\subsection{Acknowledgements}

The author was partialy supported by a Research Training Grant, DMS- 1344994, from the NSF, and would like to thank Robert Rumely and Rob Benedetto for helpful conversations. The author would also like to thank the anonymous referee for helpful feedback on earlier drafts of this paper. 

\section{Conventions and Notation}\label{sect:convandnot}

In this section we review some of the basic results about the dynamics of a rational map on $\pberk$; a rigorous development can be found in \cite{BR}. 

\subsection{Formal Structure of $\pberk$}
One way to obtain $\pberk$ is by gluing together two copies of the affine Berkovich line over $K$, denoted $\aberk$. The space $\aberk$ is the collection of equivalence classes of multiplicative seminorms on $K[T]$ that extend the absolute value $\lvert\cdot\rvert$ on $K$. Among these there are two evident seminorms: the evaluation seminorm, given $[f]_a:= |f(a)|$ for a fixed $a\in K$, and disc seminorms, $[f]_{D(a,r)}:= \sup_{x\in D(a, r)}|f(x)|$.  Any equivalence class of seminorms $[\cdot]_x$ can be obtained as a limit of disc seminorms associated to a nested, decreasing sequence of discs $D(a_i, r_i)\supseteq D(a_{i+1}, r_{i+1})\supseteq...$ (see \cite{BR} Theorem 2.2, which is a generalization of Berkovich's Classification Theorem, \cite{Ber}); more concretely: $$[f]_x := \lim_{i\to\infty} [f]_{D(a_i, r_i)}\ .$$ The resulting seminorm falls into one of four classes:
\begin{itemize}
\item[-] Type I points correspond to the evaluation seminorms described above. In this way, we often consider them as the points of $K$ lying in $\aberk$.
\item[-] Type II points correspond to disc seminorms whose discs have radius $r\in |K^\times|$.
\item[-] Type III points correspond to disc seminorms whose discs have radii that do not lie in $|K^\times|$.
\item[-] Type IV points serve to `complete' the space in some sense; they correspond to sequences of discs whose intersection is empty. $\aberk$ will not have any type IV points if the field $K$ is spherically complete.
\end{itemize}

Type II and type III points are denoted by $\zeta_{D(a, r)}$ or $\zeta_{a,r}$ where $D(a, r)$ is the associated disc. Among these, we distinguish the point $\zetaG := \zeta_{D(0,1)}$ corresponding to the unit disc in $K$; the associated seminorm is the classical Gauss norm on $K[T]$ and so $\zetaG$ is referred to as the Gauss point. We define the space $\hberk$, the Berkovich hyperbolic space, to be the collection of type II, III and IV points. 

A final fact of fundamental importance is that the action of a non-constant rational map on $\mathbb{P}^1(K)$ extends naturally to an action on $\pberk$, and such maps will preserve the type of the point upon which they act (\cite{BR} Proposition 2.15). In particular, the maps $\gamma\in \PGL_2(K)$ act transitively on the type II points of $\pberk$, and any type II point $\zeta_{a,r}$ can be written as $\zeta_{a,r} = \gamma(\zetaG)$, where $\gamma = \left(\begin{matrix} q & a \\ 0 & 1\end{matrix}\right)$ and $|q|_v = r$. 

\subsection{Resultants and Reductions}
Let $\phi\in K(z)$ be a rational map of degree $d\geq 2$, and let $\Phi = [F,G]$ be a homogeneous lift of $\phi$ to $\mathbb{A}^2$. More concretely, we have $F(X,Y) = a_dX^d+a_{d-1}X^{d-1}Y+...+a_0 Y^d$ and $G(X,Y) = b_d X^d+b_{d-1}X^{d-1}Y + ... + b_0 Y^d$, with $\phi(z) = F(z,1)/G(z,1)$. Following Rumely \cite{Ru1, Ru2}, we will say that a homogeneous lift $[F,G]$ is normalized if $F,G\in \mathcal{O}_K[X,Y]$, and at least one coefficient of $F$ or $G$ is a unit. By rescaling a given pair of polynomials $F, G$, we can always assume that $[F,G]$ is normalized. Note that a normalized lift is unique up to scaling the polynomials $F, G$ by a unit in $\mathcal{O}_K$. \\

If $[F,G]$ is a normalized lift of $\phi$, we can form the reduction of $\phi$ as follows: let $\tilde{F}(X,Y) = \tilde{a_d} X^d + ... + \tilde{a_0} Y^d , \tilde{G}(X,Y) = \tilde{b_d} X^d + ... + \tilde{b_0} Y^d$, where $\tilde{\cdot}:\mathcal{O} \to \mathcal{O}/\mathfrak{m} = k$ is the quotient map. While the maps $F,G$ were assumed to be coprime, their reductions $\tilde{F}, \tilde{G}$ need not be coprime; let $\tilde{A} = \gcd(\tilde{F}, \tilde{G})$. Writing $\tilde{F} = \tilde{A}\cdot\tilde{F}_0, \tilde{G} = \tilde{A} \cdot \tilde{G}_0$, the reduction of $\phi$ is defined to be the map on $\PP^1(k)$ represented by $[\tilde{F}_0, \tilde{G}_0]$, which we denote by $\tilde{\phi}$. 

In general, the degree of $\tilde{\phi}$ may be less than the degree of $\phi$, a reflection of the fact that $\tilde{F}, \tilde{G}$ may have factors in common that were not common to $F$ and $G$. The map $\phi$ is said to have good reduction if $\tilde{\phi}$ has the same degree as $\phi$. The map $\phi$ is said to have potential good reduction if, after a change of coordinates by $\gamma \in \PGL_2(K)$ -- i.e. replacing $\phi$ by $\phi^\gamma = \gamma^{-1} \circ \phi \circ \gamma$  -- the resulting map $\phi^\gamma$ has good reduction. If neither of these cases hold, we say that $\phi$ has bad reduction.\\

We can, more generally, speak of the reduction of $\phi$ at a type II point $P\in \hberk$ as follows: choose $\sigma_1, \sigma_2 \in \PGL_2(K)$ so that $\sigma_1(\zetaG) = P$ and $\sigma_2(\zetaG) = \phi(P)$. Then the reduction of $\phi$ at $P$ is defined to be the reduction of a normalized lift of $\phi^\sigma = \sigma_2\circ\phi\circ\sigma_1$. We denote the reduction of $\phi$ at $P$ again by $\tilde{\phi}$, letting the context determine the point at which the reduction is being considered. The degree of the reduction at $P$ will be written $\deg_P(\phi)$. Following the definition above, we say that $\phi$ has good reduction at $P$ if $\deg_P(\phi)$ equals the degree of $\phi$.

A type II point $P\in \hberk$ with $\deg_P(\phi) = 1$ is called an indifferent point. In \cite{Ru2}, Rumely introduced a further stratification of the reduction of $\phi$ at indifferent points:
\begin{itemize}
\item If, after some change of coordinates on $\mathbb{P}^1(k)$ the reduction $\tilde{\phi}$ lifts to a map of the form $\tilde{\Phi}=[\lambda X, Y]$ for some $\lambda\in k\setminus \{0, 1\}$, we say that $\Phi$ has multiplicatively indifferent reduction at $P$.
\item If, after change of coordinates on $\PP^1(k)$, $\tilde{\phi}$ has a lift of the form $\tilde{\Phi} = [X+aY, Y]$ for some $a\in k\setminus \{0\}$ we say that $\phi$ has additively indifferent reduction. 
\item If $\tilde{\phi}$ lifts to $\tilde{\Phi} = [X,Y]$, then we say that $\phi$ has id-indifferent reduction at $P$. 
\end{itemize}
The reduction type of indifferent points affects the behaviour of $\phi$ nearby those points. We also briefly recall the {\em locus of id-indifference}, denoted $U_\textrm{id}$. This is the set of all point $P\in \hberk$ for which $\phi$ has id-indifferent reduction\footnote{ Technically, one must also define the reduction types for type III and IV points, and include the resulting id-indifferent points. This is done by passing to an appropriate extension of $K$; see \cite{Ru2} Section 9.}. Details about this set are given in Sections 9 and 10 of \cite{Ru2}; here we recall that $U_\textrm{id}$ has at most finitely many connected components, and the closure of each component contains at least two type I fixed points (counting multiplicity). Endpoints of $U_\textrm{id}$ are either repelling type II fixed points, additively indifferent type II fixed points, or indifferent type I fixed points.\\

A way to measure whether or not two homogeneous polynomials $F(X,Y), G(X,Y)$ have a common factor is by looking at the resultant. It is a polynomial in the coefficients of $F$ and $G$ that vanishes precisely when $F$ and $G$ have a common factor. Formally, it is defined as follows: let $[F,G]$ be a normalized lift of $\phi$. Then

\begin{align*}
\ord\Res(F,G) = & \ord\det\left(
\begin{array}{cccccccc} 
a_d & a_{d-1} & \dots & a_1 &a_0 &0 &  \dots & 0 \\ 
0 & a_d & a_{d-1} & \dots & a_1 & a_0 & \dots & 0 \\
\vdots&& & \ddots & \vdots & \vdots & & \vdots \\
0 & 0 & 0 & a_d & a_{d-1} & \dots & a_1 & a_0 \\
b_d & b_{d-1} & \dots & b_1 &b_0 &0 &  \dots & 0 \\ 
0 & b_d & b_{d-1} & \dots & b_1 & b_0 & \dots & 0 \\
\vdots&& & \ddots & \vdots & \vdots & & \vdots \\
0 & 0 & 0 & b_d & b_{d-1} & \dots & b_1 & b_0 \\
\end{array}\right)\ .
\end{align*}

We observe that $\ordRes(cF,cG) = \ordRes(F,G) + 2d\ord(c)$, and hence the quantity $\ordRes(F,G)$ is independent of which normalized lift $[F,G]$ of $\phi$ we choose. Note also that for a normalized representation $[F,G]$ of $\phi$, we have $\widetilde{\Res(F,G)} = \Res(\tilde{F}, \tilde{G})$. It follows that $\phi$ will have good reduction if and only if $\ord\Res(F,G) = 0$ for any choice of normalized representation $[F,G]$. Moreover, $\phi$ will have potential good reduction if and only if, for some $\gamma$ and some normalized representation $[F^\gamma, G^\gamma]$ of $\phi^\gamma$, we have $\ord\Res(F^\gamma, G^\gamma) = 0$. 

This allowed $\ordRes$ to be considered as a function on $\hberk$: given a type II point $\zeta_{a,r}\in \hberk$, let $\gamma = \left(\begin{matrix} q & a \\ 0 & 1 \end{matrix}\right)\in \PGL_2(K)$, where $|q|_v = r$. Rumely defines (see \cite{Ru1}) $$\ordRes_\phi(\zeta_{a,r}) := \ordRes\left(F^{\gamma_{a,r}}, G^{\gamma_{a,r}}\right)\ .$$ A priori, this is only defined on type II points; however, Rumely shows (\cite{Ru1} Theorem 0.1) that this extends to a continuous function on all of $\pberk$ which is piecewise affine along segments in $\hberk$. 

\subsection{Topologies on $\pberk$}
The space $\pberk$ carries two natural topologies. The first is the weak, or Gelfand, topology. In this topology, $\pberk$ is locally compact and Hausdorff, but in general will not be metrizable. The second topology, called the strong topology, is generated by a metric $\rho$, but when $K$ is algebraically closed, $\pberk$ fails to be locally compact in this topology. 

In both topologies, $\pberk$ is path connected, and in fact it is uniquely path connected. This is most readily seen by observing that $\pberk$ can be given the structure of an $\RR$-tree: edges in $\hberk$ are homeomorphic to real intervals via maps of the form $[r,s]\to \hberk$ given by $t\mapsto \zeta_{a, q_v^{-t}}$ for some $a\in K$. The metric $\rho$ that induces the strong topology is defined so that each such map an isometry. For more details, see \cite{BR} Chapter 2. The type I and type IV points form the endpoints of the $\RR$-tree $\pberk$. 

Given a point $P\in \hberk$, we will denote by $\B_\rho(P,r)$ the collection of points $Q$ such that $\rho(P,Q) < r$. If $V$ is a subset of $\pberk$ that is closed in the strong topology, we will let $B_\rho(V,r) = \{x\in \pberk\ : \ \inf_{v\in V} \rho(x,v) < r\}$; in a similar manner, $\rho(x, V) = \inf_{v\in V} \rho(x,v)$ is distance between $x$ and the nearest point of $V$. \\

The tree structure of $\pberk$ allows us to introduce the notion of a tangent space at a point $P\in \pberk$, which we will denote $T_P$. Formally, the tangent space at $P$ is collection of equivalence classes of paths $(P, Q_0]$ emanating from $P$, where two paths are equivalent if they share a common initial segment. For $\vv\in T_P$ and fixed small values of $t>0$, the expression $P+t\vv$ will mean the following: choose a point $Q_0$ for which $(P, Q_0]$ is in the equivalence class for $\vv$ and $\rho(P, Q_0) >t$; then $P+t\vv$ is the unique point $R\in (P, Q_0]$ with $\rho(P,R) = t$. While this definition techinally depends on our choice of $Q_0$, typical applications involving this notation are independent of which $Q_0$ is chosen (e.g., if we let $t\to 0$). Situations requiring a more specific choice of path will be handled individually.

The tangent directions $\vv\in T_P$ at type II points $P$ are in one-to-one correspondence with the points of $\mathbb{P}^1(\tilde{k})$ (this is canonical only up to a choice of coordinates for $\mathbb{P}^1(k)$). For type III points, $T_P$ contains two directions (one towards infinity, the other away from infinity), while for type I and type IV points P, $T_P$ is a single direction pointing into $\hberk$. 

The tangent directions can also be used to parameterize connected components of $\pberk \setminus \{P\}$; we will denote by $\B_{P}(\vv)^-$ the connected component of $\pberk \setminus \{P\}$ containing the points $P+t\vv$ for small values of $t>0$. This should not be confused with $\B_{\rho}(P,r)$ introduced above, which instead denotes the ball of $\rho$-radius $r$ about $P$.

Fix a type II point $P\in \pberk$. If a  system of coordinates on $\PP^1(k)\simeq T_P$ is given, we may write $\vv_{\tilde{a}}$ for the direction corresponding to $\tilde{a}\in \PP^1(k)$. Alternatively, $Q\in \pberk\setminus\{P\}$, we may also write $\vv_Q\in T_P$ for the direction satisfying $Q\in B_{P}(\vv_Q)^-$; thus the tilde in the subscript of $\vv$ will be important. Finally, in a few places we will write $\vv_1, ..., \vv_N$ to mean a finite list of $N$ tangent vectors at $P$, as opposed to the directions towards 1, 2, ..., N; the context will make clear when this is the case.

Frequently we will study finite, connected subgraphs $\Gamma$ of $\hberk$: these are subtrees of $\hberk$ with finitely many edges each with finite length. We can extend the notion of tangent space given above to the notion of the tangent space at $P$ in $\Gamma$, the collection of those equivalence classes of paths having an initial segment lying in $\Gamma$. We denote this space by $T_P (\Gamma)$, and its cardinality is the {\em valence} of $\Gamma$ at $P$, denoted by $v_\Gamma(P)$. An important class of functions defined on such graphs are those which are continuous and piecewise affine along the branches of $\Gamma$; that is, for such $f$ there exists a finite set $\{s_1, ..., s_n\}\subseteq\Gamma$ such that $\Gamma\setminus \{s_1, ..., s_n\}$ is finite collection of segments each isometric to an open interval in $\mathbb{R}$, and $f$ is continuous on $\Gamma$ and affine on the components of $\Gamma \setminus \{s_1, ..., s_n\}$. We denote the space of such functions by CPA($\Gamma$).\\

By the unique path connectedness of $\pberk$, one can also introduce the notion of a retraction map from one subset to another. If $U, V\subseteq \pberk$ are path connected subsets, and $V$ is closed (in either the weak or the strong topology), then we can define a retraction $r_{U, V}: U\to V$ by fixing $v\in V$ and sending each point $x\in U$ to the first point on $[x, v]$ that intersects $V$. That this map is well defined (independent of choice of $v\in V$) follows from the unique path connectedness of $\pberk$. Most often we will consider retractions $r_{\pberk, \Gamma}$ where $\Gamma$ is a finite, connected subtree of $\pberk$; these maps we will denote simply by $r_{\Gamma}$. The retraction maps will be of fundamental importance in constructing the Laplacian of a map on $\pberk$.

Finally, we will make use of a type of `supremum' on $\pberk$, defined as follows: let $P, Q, R\in \pberk$, and consider the segments $[P, R],[Q,R]$; then $P\wedge_R Q$ is the point in $[P, R]\cap [Q,R]$ that is furthest from $R$. This intersection is always nonempty, since $R\in [P,R]\cap [Q,R]$.

\subsection{Laplacians and Potential Theory on $\pberk$}
The theory of Laplacians on $\pberk$ is based on the theory of Laplacians for finite connected graphs; see, e.g. \cite{BRHarmonic} and \cite{CR}. 

Let $\Gamma\subseteq \hberk$ be a finite connected subtree, and fix $f\in $CPA($\Gamma$). For each $P\in \Gamma$ and each direction $\vv\in T_P(\Gamma)$, we can define the slope of $f$ at $P$ in the direction $\vv$ as $$\del_v(f)(P) = \dlim_{t\to 0} \dfrac{f(P+t\vv)-f(P)}{t}\ .$$ For $f\in \CPA(\Gamma)$, this limit always exists, though it may not exist for more general functions. The Laplacian of $f$ on $\Gamma$ is then defined to be the measure $$\Delta_{\Gamma}(f) := -\sum_{P\in \Gamma} \sum_{v\in T_P \Gamma} \del_v(f)(P)\delta_P\ ,$$  where $\delta_P$ denotes the Dirac point mass at the point $P$. This notion can be extended, both to more general classes of functions and to more general subsets of $\pberk$. On a domain $U\subseteq\pberk$, the largest class of functions on which a Laplacian can be defined is called the space of functions of `bounded differential variation', which is denoted $\BDV(U)$; intuitively, these functions to do not `wiggle' more than they should along a given path. A fundamental property is that Laplacians defined on larger spaces must be compatible with the retraction\footnote{ See \cite{BR} Section 5.2; this is essentially how the Laplacian is defined on arbitrary domains.}; namely, if $U\subseteq\pberk$ is closed, $\Gamma\subseteq U$, and $f\in \BDV(U)$, then $$\Delta_{\Gamma} = (r_{U, \Gamma})_* \Delta_{U}$$ where $(r_{U, \Gamma})_*$ denotes the pushforward of the Laplacian on $U$ given by $(r_{U, \Gamma})_* \Delta_U (B) = \Delta_U(r_{U, \Gamma}^{-1} B)$ for every Borel set $B\subseteq \Gamma$. 

Let $\nu$ be a probability measure on $\pberk$: a positive \emph{Radon} measure with total mass 1. For a fixed $\zeta\in \pberk$, the potential function associated to $\nu$ is given $$u_\nu(z, \zeta) = \int_{\pberk} -\log_v \delta(z, w)_{\zeta} d\nu(w)\ .$$ Here, $\delta(z, w)_\zeta$ is the Hsia kernel relative to $\zeta$; when $\zeta=\infty$ this is an extension of the usual absolute value on $K$ to $\pberk$, and when $\zeta = \zetaG$ this is an extension of the chordal metric $||\cdot, \cdot||$ on $\PP^1(K)$ to $\pberk$. See \cite{BR} Chapter 4 for a detailed construction of the Hsia kernel.

We say that $\nu$ has continuous potentials if for some fixed $\zeta\in \hberk$, the function $u_\nu(z, \zeta)$ is continuous in the weak topology. Necessarily if $u_\nu(z,\zeta)$ is continuous for one $\zeta\in \hberk$, then it is continuous for any fixed $\zeta_0\in \hberk$ (see the discussion following Definition 5.40 in \cite{BR}). We will say that $\nu$ has bounded potentials if, for some fixed $\zeta\in \hberk$, the function $u_\nu(z, \zeta)$ is bounded. Since $\pberk$ is compact, a measure with continuous potentials necessarily has bounded potentials. See \cite{BR}, Chapter 5 for a detailed discussion of the Laplacian and \cite{BR}, Chapter 6 for a rigorous development of potential functions on $\pberk$.

\subsection{Arakelov-Green's Functions}
Let $\nu$ be a Radon probability measure on $\pberk$. The Arakelov-Green's function attached to $\nu$ is given \begin{align}\label{eq:firstgreendecomp}g_{\nu}(x,y) = \int_{\pberk} -\log_v \delta(x,y)_\zeta d\nu(\zeta)+C\ ,\end{align} where $C$ is a constant chosen to ensure $$\iint g_{\nu}(x,y)d\nu(x)d\nu(y) = 0\ .$$ When $\nu=\mu_{\phi}$ is the invariant measure associated to a rational map $\phi$, and $x,y\in \hberk$, a fundamental result (see \cite{BR}, Theorem 10.21 and the discussion following) is that $g_{\mu_\phi}(x,y)$ admits a decomposition as $$g_{\mu_{\phi}}(x,y) = -\log(\delta(x,y)_{\infty}) +\hat{h}_{\phi}(x) + \hat{h}_{\phi}(y) - \dfrac{1}{d^2-d}\log(|\Res(F,G)|)\ .$$ Here $\hat{h}_{\phi}$ is the Berkovich canonical height attached to $\phi$ given in $\aberk$ by (see \cite{BR}, Chapter 10) \begin{equation}\label{eq:height} \hat{h}_{\phi, \infty}(x) = \lim_{n\to\infty} \frac{1}{d^n} \log_v \max \left([F^{(n)}(T,1)]_x, [G^{(n)}(T,1)]_x\right)\ .\end{equation}  We will be interested in the diagonal values of the Arakelov-Green's function, and so we will mostly consider \begin{align}\label{eq:greendecompintro}g_{\mu_{\phi}}(x,x) = -\log(\delta(x,x)_{\infty}) + 2\hat{h}_{\phi}(x) -\dfrac{1}{d^2-d}\log|\Res(F,G)|\ .\end{align}

Following Baker and Rumely, we will often denote $g_{\mu_\phi}(x,y)$ by $g_\phi(x,y)$ (\cite{BR} Chapter 10).

\subsection{Convergence of Closed Subsets of $\pberk$}
Let $X$ be any Hausdorff topological space, and let $$\textrm{CL}(X):=\{A\subseteq X\ : \ A \textrm{ is closed.}\}\ .$$ If $(X,d)$ is a metric space, we can equip $\textrm{CL}(X)$ with the Hausdorff metric: for $A, B \in \textrm{CL}(X)$, let $$H_d(A,B):= \max\left(\sup_{x\in A} \inf_{y\in B} \rho(x,y), \sup_{x\in B}\inf_{y\in A} \rho(x,y)\right)\ .$$ Note that the Hausdorff metric need not be finite, but its restriction to closed and bounded subsets will be finite. The metric space $(\textrm{CL}(X), H_d)$ is complete if and only if $(X, d)$ is complete. See \cite{Beer} for a more thorough discussion of the Hausdorff metric and the topology it generates on $CL(X)$.

\section{Convergence of the functions $\dfrac{1}{d^{2n}-d^n}\ord\Res_{\phi^n}(x)$}\label{sect:fnconv}

In this section we prove Theorem \ref{thm:fnconv} by deriving the following more explicit estimate of convergence:

\begin{thm}\label{thm:fnconvexplicit}
Let $K$ be a complete, non-Archimedean valued field, and let $\phi\in K(z)$ have degree $d\geq 2$. There is a constant $C= C(\phi) >0$ depending only on $\phi$ such that for any $x\in \hberk$, we have
\begin{align*}
\left| \dfrac{1}{d^{2n}-d^n} \ord\Res_{\phi^{(n)}}(x) - g_{\mu_{\phi}}(x,x)\right| &\leq \dfrac{4}{d^n-1} \max\left(C,\rho(x, \zetaG)\right)\ .
\end{align*}
\end{thm}

 The motivation for the proof is a similarity between a decomposition of $\ord\Res_{\phi}(x)$ given in \cite{Ru1} and the decomposition of $g_{\mu_{\phi}}(x,x)$ given above in (\ref{eq:greendecompintro}). The similarity is summarized in Table \ref{table:decomp} below. 

\subsection{Decompositions of $\ord\Res_{\phi}(x)$ and $g_{\mu_{\phi}}(x,x)$}\label{ssect:prepfnconv}
We begin with the decomposition of $\ord\Res_{\phi}(x)$ given in \cite{Ru1}. Let $\zeta \in \hberk$ be a type II point, and let $\gamma\in GL_2(K)$ be an element such that $\gamma(\zetaG) = \zeta$. As above we let $\Phi$ be a lift of $\phi$ to $\mathbb{A}^2(K)$ with $\Phi(X,Y) = [F(X,Y),G(X,Y)]$, where $F(X,Y), G(X,Y)$ are assumed to be normalized. 

Write $F(X,Y) = a_dX^d+a_{d-1}X^{d-1}Y+...+a_0 Y^d$ and $G(X,Y) = b_d X^d+b_{d-1}X^{d-1}Y + ... + b_0 Y^d$. In a similar manner let $F^{\gamma}, G^\gamma$ denote the components of a normalized lift of $\phi^\gamma$, with coefficeints $a_i^\gamma, b_i^\gamma$ respectively. By direct computation of the resultant, we have 
\begin{align}\label{eq:ORdecomp}
\ord\Res_{\phi}(\zeta) = \ord\Res(F,G) + (d^2+d)\ord(\det(\gamma)) - 2d\min(\ord(F^{\gamma}), \ord(G^{\gamma}))\ ,
\end{align}
 where $\ord(F) = \min_{0\leq i \leq d} (\ord(a_i))$ and similarly for $G, F^\gamma, G^\gamma$ (this is \cite{Ru1} Formula (8)). For our purposes, we record an iterated version of this equation. Let $\Phi^{(n)} = [F^{(n)}(X,Y), G^{(n)}(X,Y)]$ be a normalized homogeneous lift of $\phi^n$ to $\mathbb{A}^2(K)$. Then Equation (\ref{eq:ORdecomp}) becomes

\begin{align}\label{eq:ORndecomp}
\ord\Res_{\phi^n}(\zeta) = \ord\Res\left(F^{(n)},G^{(n)}\right) &+ (d^{2n}+d^n)\ord(\det(\gamma)) \\&- 2d^n\min\left(\ord\left((F^{(n)})^{\gamma}\right), \ord\left((G^{(n)})^{\gamma}\right)\right)\nonumber\ .
\end{align}

There is a correspondence between the terms appearing in the decomposition of $\ordRes_\phi$ given in (\ref{eq:ORndecomp}) and the decomposition of $g_\phi(x,x)$ given in (\ref{eq:greendecompintro}) that is summarized in the following table:\\

\begin{table}[h]
\centering
\begin{tabular}{c|c}
$\ord\Res_{\phi^{(n)}}(\zeta)$ & $g_{\mu_\phi}(\zeta,\zeta)$\\
\hline
\hline
$\ord\Res(F^{(n)},G^{(n)})$ & $- \dfrac{1}{d(d-1)} \log_v ( |\Res(F, G)|) $\\[.2cm]
\hline
$(d^{2n}+d^n) \ord(\det(\gamma))$ & $-\log_v(\delta(\zeta,\zeta)_{\infty})$\\[.2cm]
\hline
$-2d^n\min(\ord((F^{(n)})^{\gamma}), \ord((G^{(n)})^{\gamma}))$ & $2\hat{h}_{\phi, v, (\infty)} (\zeta) $ \\[.2cm]
\end{tabular}
\caption{ Comparison of Decompositions}
\label{table:decomp}
\end{table}
Heuristically, we show that each term in the left column, when normalized by $\frac{1}{d^{2n}-d^n}$, converges to the corresponding term on the right side (this only works heuristically, and we will in fact need to work with Rows 2 and 3 as a single unit rather than treating them separately).

\subsection{Preparatory Results}
The convergence of the terms in the first row of Table~\ref{table:decomp} is straightforward:

\begin{lemma}\label{lem:ordresit}
For every $n$, we have 
$$\dfrac{1}{d^{2n}-d^n} \ord\Res\left(F^{(n)}, G^{(n)}\right) = -\dfrac{1}{d^2-d} \log_v\left|\Res(F,G)\right|\ .$$
\end{lemma}
\begin{proof}
Using the formula for the resultant of a composition given in \cite{ADS} Exercise 2.12(a), we find
$$\Res(F^{(n)}, G^{(n)}) = \Res(F, G)^{d^{n-1}} \Res(F^{(n-1)}, G^{(n-1)})^{d^2}\ .$$
Applying this inductively, 
\begin{align*}
\Res(F^{(n)}(X,Y), G^{(n)}(X,Y)) & = \Res(F, G)^{d^{n-1} + ... + d^{2n-2}}\\
& = \Res(F, G)^{d^{n-1} (1+d + ... + d^{n-1})} \\
& =\Res(F, G)^{d^{n-1} \frac{d^n -1}{d-1}}\\
& = \Res(F, G)^{\frac{d^{2n} - d^{n}}{d(d-1)}}\ .
\end{align*}

\noindent Taking the ord and normalizing, we obtain the result
\begin{align*}
\dfrac{1}{d^{2n}-d^n}\ord\Res(F^{(n)}, G^{(n)}) & = \dfrac{1}{d^{2n}-d^n} \ord\left(\Res(F,G)^{\frac{d^{2n}-d^n}{d^2-d}}\right)\\
& = \dfrac{1}{d^2-d} \ord\Res(F,G)\\
& = -\dfrac{1}{d^2-d} \log_v \left|\Res(F,G)\right|\ .
\end{align*}
\end{proof}

We now turn to the convergence of the terms in the second and third rows of Table~\ref{table:decomp}. The terms on the second line of Table~\ref{table:decomp} are related by the following lemma:

\begin{lemma}\label{lem:orddethsia} If $x$ is the type II point $\zeta_{a, r}\in\hberk$, then the transformation $\gamma\in \PGL_2(K)$ given $\gamma(z) = bz+a$, where $|b| = r$, sends $\zetaG$ to $x$, and we have
$$\ord(\det(\gamma)) = -\log_v(\delta(x,x)_{\infty})\ .$$
\end{lemma}
\begin{proof}
Let $\gamma$ be as in the statement of the lemma; as a matrix, $\gamma$ is represented by $\left[\begin{smallmatrix} b & a \\ 0 & 1 \end{smallmatrix} \right]$; clearly $x=\gamma(\zetaG)$. Since $x$ corresponds to a disk of radius $r$ and $\delta(x,x)_{\infty}=\textrm{diam}_{\infty}(x)=r$, we have
\begin{align*}
-\log_v (\delta(x,x)_{\infty}) &= -\log_v \left|b\right|\\
& = \ord (b)\ .
\end{align*} Note that $\det(\gamma) = b$, and so $\ord(\det(\gamma)) = \ord(b) = -\log_v(\delta(x,x)_{\infty})$.
\end{proof}

Next we look to compare the terms $$-2d^n \min\left(\ord\left((F^{(n)})^{\gamma}\right), \ord\left((G^{(n)})^{\gamma}\right)\right)$$ and $2\hat{h}_{\phi, v}$. Let $|F(X,Y)| = \max_{1\leq i \leq d} |a_i|$ denote the absolute value of the largest coefficient of $F(X,Y)$. We can rewrite the above expression in terms of a log max of the absolute values:
\begin{align}
-2d^n \min\left(\ord\left((F^{(n)})^{\gamma}\right), \ord\left((G^{(n)})^{\gamma}\right)\right) = 2d^n \log \max \left( \left| (F^{(n)})^\gamma\right|, \left|(G^{(n)})^\gamma\right|\right)\ .
\end{align}

In the discussion below, we make two simplifying assumptions. First, we will only work with affine transformations $\gamma(z) = az+b$; these maps can be used to carry $\zetaG$ to any type II point in $\hberk$, and so will be sufficient for our purposes. Second, rather than working with iterates $F^{(n)}, G^{(n)}$, we will simply state and prove our results for arbitrary homogeneous polynomials $F, G\in K[X, Y]$.

The expression for the conjugate $\Phi^\gamma$ can be given

\begin{align*}
\left[ \begin{array}{c} F^{\gamma}(X,Y) \\ G^{\gamma}(X,Y) \end{array}\right] &= \left(\Adj(\gamma)\cdot \left[ \begin{array}{c} F \\ G \end{array} \right]\right) \left(\gamma\left(\left[ \begin{array}{c} X \\ Y \end{array}\right]\right)\right)\\[4pt]
& = \left(\left[ \begin{matrix} 1 & -a \\ 0 & b \end{matrix}\right] \cdot \left[ \begin{array}{c} F \\ G \end{array} \right]\right) \left( \left[ \begin{matrix} b & a \\ 0 & 1 \end{matrix} \right] \cdot \left[ \begin{array}{c} X \\ Y \end{array}\right]\right)\\[0.2cm]
& = \left[\begin{array}{c} F(bX+aY, Y) -a G(bX+aY, Y) \\b G(bX+aY, Y) \end{array}\right]\ .
\end{align*}
 It is worth noting here that $[F^\gamma, G^\gamma]$ may not be normalized!

We will address the relation between the coefficients of $[F^{\gamma}, G^{\gamma}]$ and $[F,G]$ in two steps, first looking at the effect of postcomposition by $\Adj(\gamma)$, and then the effect of precomposition by $\gamma$.

\begin{lemma}\label{lem:precomp}
Let $F(X,Y), G(X,Y)$ be a pair of homogeneous degree $d$ polynomials in $K[X,Y]$. For $a,b \in K$, if $x=\zeta_{a, |b|}\in \hberk$, we have
\begin{align}
\big| \log \max \big( \left|F(X,Y)-aG(X,Y)\right|, \left |bG(X,Y)\right|\big) - \log \max \big(\left|F(X,Y)\right|, \left|G(X,Y)\right|\big)\big|\leq \rho(x, \zetaG)\ .
\end{align} 
\end{lemma}

\begin{proof}
The result follows from explicit estimates on the coefficients based on the proof of \cite{ADS} Theorem 3.11.

Write $F(X,Y) = a_dX^d + a_{d-1}X^{d-1}Y+...+a_0 Y^d$, $G(X,Y) = b_d X^d+...+b_0Y^d$. The coefficients of $F(X,Y) - aG(X,Y)$ are of the form $a_i-a\cdot b_i$, and likewise the coefficients of $bG(X,Y)$ are $b\cdot b_i$. By the ultrametric inequality, we obtain estimates towards the lower bound by:
\begin{align*}
|a_i-a\cdot b_i| & \leq \max \left( |a_i|, |b_i|\right) \cdot \max \left( 1, |a| \right)\ ,\\
|b\cdot b_i| & \leq \max \left(|a_i|, |b_i| \right)\cdot \max(1, |b|)\ .
\end{align*} Hence 
\begin{align*}
\max\left(|a_i-a\cdot b_i|, |b\cdot b_i|\right) \leq \max(|a_i|, |b_i|) \cdot \max(1, |a|, |b|)\ ,
\end{align*} or equivalently
\begin{align}\label{eq:lowerbd}
\frac{\max\left(|a_i-a\cdot b_i|, |b\cdot b_i|\right)}{\max(|a_i|, |b_i|)|}\leq \max(1, |a|, |b|)\ .
\end{align}Similarly, for the upper bound, we have
\begin{align*}
|a_i | = |a_i - a b_i + a b_i|  & \leq \max \left( |a_i - a\cdot b_i|, |b\cdot b_i| \right) \max \left( 1, \frac{|a|}{|b|}\right)\ ,\\
|b_i|  = \frac{1}{|b|} |b b_i| & \leq \max \left( |a_i - a\cdot b_i|, |b\cdot b_i| \right) \max \left(1, \frac{1}{|b|}\right)\ .
\end{align*} Hence
\begin{align*}
\max(|a_i|, |b_i|) \leq \max(|a_i-a\cdot b_i|, |b\cdot b_i|) \cdot \max\left(1, \frac{|a|}{|b|}, \frac{1}{|b|}\right)\ ,
\end{align*} or equivalently
\begin{align}\label{eq:upperbd}
\frac{\max(|a_i|, |b_i|)}{\max(|a_i-a\cdot b_i|, |b\cdot b_i|) }\leq\max\left(1, \frac{|a|}{|b|}, \frac{1}{|b|}\right)\ .
\end{align} Combining (\ref{eq:lowerbd}) and (\ref{eq:upperbd}), taking logs, and doing some algebra yields:
\begin{align*}
\big| \log \max \left( |a_i-a\cdot b_i|, |b \cdot b_i|\right) - & \log \max \left(|a_i|, |b_i|\right)\big|\\
& \leq \log\max\left(\max\left(1, \frac{|a|}{|b|}, \frac{1}{|b|}\right), \max(1, |a|, |b|)\right)\\
& \leq \log \max\left(1, \frac{|a|}{|b|},\frac{1}{|b|}\right) + \log_v \max (1, |a|, |b|)\\
& =2 \log \max(1, |a|, |b|) - \log(|b|)\ .
\end{align*} Finally, we note that if $x=\zeta_{a, |b|}$, the smallest disc containing $D(a, |b|)$ and $D(0,1)$ has radius $R=\max(|a|, |b|, 1)$; hence $x\wedge_\infty \zetaG = \zeta_{0, R}$. The above estimate now reads
\begin{align*}
\big| \log \max \left( |a_i-a\cdot b_i|, |b \cdot b_i|\right) - & \log \max \left(|a_i|, |b_i|\right)\big|\\
& \leq 2\log \max(1, |a|, |b|) - \log(|b|)\\
& = 2\log R - \log(|b|)\\
& = \rho(x, \zetaG)\ .
\end{align*}
\end{proof}

We now have a lemma that makes explicit the effect of precomposition of $[F,G]$ by $\gamma$:

\begin{lemma}\label{lem:postcomp}
Let $F(X,Y), G(X,Y)$ be a pair of homogeneous degree $d$ polynomials in $K[X,Y]$. For $a,b\in K$, if $x=\zeta_{a, |b|}\in \hberk$, we have $$\log \max(|F(bX+aY, Y)|, |G(bX+aY, Y)|) = \log \max( [F(T,1)]_x, [G(T,1)]_x)\ ,$$ where $[F(T,1)]_x$ denotes the $($semi$)$norm corresponding to $x$.
\end{lemma}

\begin{proof}
First recall that the norm induced by the Gauss point is indeed the Gauss norm: $[F(T,1)]_{\zeta_{Gauss}} = \max_{0\leq i \leq d}(|a_i|) = |F(T,1)|$.

Let $T=X/Y$. We have $F(bX+aY, Y) = \frac{1}{Y^d}F(bT+a, 1)$, and since the division by $Y$ does not affect the maximum of the coefficients, we have
\begin{align*}
|F(bX+aY, Y)| & = |F(bT+a, 1)| \\
&= |F(\gamma(T), 1)| \\
& = [F(\gamma(T), 1)]_{\zeta_{Gauss}} \\
& = [F(T,1)]_{\gamma(\zeta_{Gauss})} \\
& = [F(T,1)]_x\ .
\end{align*} The similar statement holds for $G(X,Y)$, and so the result follows.
\end{proof}

We can combine the two preceeding lemmas to obtain a result that expresses the effect of conjugation by an affine map $\gamma$ on the size of the coefficients of a pair $[F,G]$:

\begin{lemma}\label{lem:htconvergent}
Let $F(X,Y), G(X,Y)$ be a pair of homogeneous degree $d$ polynomials in $K[X,Y]$. Let $x\in \hberk$ be of type II, and let $\gamma(z) = bz+a$ be the affine map sending $\zetaG$ to $x$. Let $\hat{\ell}^{(n)}_{\phi, \infty}(x)$ denote the (unnormalized) convergent of $\hat{h}_\phi$ given (\ref{eq:height}): $$\hat{\ell}^{(n)}_{\phi, \infty}(x) :=\log\max \left( [F^{(n)}(T,1)]_x, [G^{(n)}(T,1)]_x\right)\ .$$ Then 
\begin{align*}
\left|\log \max \left( \left|\left(F^{(n)}\right)^\gamma \right|, \left| \left( G^{(n)}\right)^\gamma \right| \right) - \hat{\ell}^{(n)}_{\phi, \infty}(x)\right|\leq \rho(x, \zetaG)\ .
\end{align*}
\end{lemma}
\begin{proof}
We first apply Lemma \ref{lem:precomp} to find 
\begin{align*}
\left| \log \max \left( \left|\left(F^{(n)}\right)^\gamma \right|, \left| \left( G^{(n)}\right)^\gamma \right| \right) - \log \max \left( \left|F^{(n)}(bX+aY, Y)\right|, \left| G^{(n)}(bX+aY, Y)\right|\right)\right|\leq \rho(x, \zetaG)\ .
\end{align*} 

\noindent Applying Lemma \ref{lem:postcomp} this becomes \begin{align*}
\left| \log \max \left( \left|\left(F^{(n)}\right)^\gamma \right|, \left| \left( G^{(n)}\right)^\gamma \right| \right) - \log \max \left( \left[F^{(n)}(T,1)\right]_x, \left[ G^{(n)}(T,1)\right]_x\right)\right|\leq \rho(x, \zetaG)\ .
\end{align*} Equivalently, 
\begin{align*}
\left| \log \max \left( \left|\left(F^{(n)}\right)^\gamma \right|, \left| \left( G^{(n)}\right)^\gamma \right| \right) - \hat{\ell}^{(n)}_{\phi, \infty}(x)\right|\leq \rho(x, \zetaG)\ .
\end{align*}

\end{proof}

The above proposition shows that the terms $\frac{1}{d^n}\log \max \left( \left|\left(F^{(n)}\right)^\gamma \right|, \left| \left( G^{(n)}\right)^\gamma \right| \right)$ behave very similarly to the convergents of $\hat{h}_{\phi,v}$ given in \cite{BR}, Equation (10.9); we make this relation precise in the following proposition:

\begin{prop} \label{prop:htconvergent}
Let $x\in \hberk$ be given by $x=\gamma(\zetaG)$, where $\gamma(z) = bz+a$. There exists a constant $C_\phi$ depending only on $\phi$ such that:
\begin{align}\label{eq:htconvergentextra}
\left|-\frac{1}{d^n-1} \min\left(\ord\left((F^{(n)})^{\gamma}\right), \ord\left((G^{(n)})^{\gamma}\right)\right) \right. & \left. - \hat{h}_{\phi, \infty}(x) - \dfrac{1}{d^n-1} \log_v(\delta(x,x)_{\infty})\right| \nonumber\\
& \leq \dfrac{2}{d^n-1}\max\left(C_\phi,\  \rho(x, \zetaG)\right)\ .
\end{align}
\end{prop}

\noindent\textbf{Remark:} There is a seemingly `extra' term $\frac{1}{d^n-1}\log_v(\delta(x,x)_{\infty})$ appearing in the left side of the inequality (\ref{eq:htconvergentextra}); this term both cleans up the proof below and facilitates the proof of Theorem \ref{thm:fnconv}.

\begin{proof}
To ease notation, let 
\begin{align*}
\hat{k}^{(n)}_{\phi}(x) &:= \log \max \left( \left|\left(F^{(n)}\right)^\gamma \right|, \left|\left(G^{(n)}\right)^\gamma\right|\right)\\
& = -\min\left(\ord\left((F^{(n)})^{\gamma}\right), \ord\left((G^{(n)})^{\gamma}\right)\right) \ .
\end{align*} The statement of Lemma \ref{lem:htconvergent} tells us that 
\begin{align}\label{eq:preppiece}
|\hat{k}^{(n)}_{\phi}(x) - \hat{\ell}^{(n)}_{\phi, \infty}(x)| \leq \rho(x, \zetaG)\ .
\end{align} We find that
\begin{align}
\left| \dfrac{1}{d^n-1} \hat{k}^{(n)}_{\phi, \infty}(x) \right.&\left.- \hat{h}_{\phi, \infty}(x) - \dfrac{1}{d^n-1} \log_v(\delta(x,x)_{\infty})\right|\nonumber\\
& \leq\frac{1}{d^n-1} |\hat{k}^{(n)}_{\phi}(x) - \hat{\ell}^{(n)}_{\phi,\infty}(x)| + \left| \frac{1}{d^n-1} \hat{\ell}^{(n)}_{\phi, \infty}(x) - \hat{h}_{\phi, \infty}(x) - \frac{1}{d^n-1} \log_v \delta(x,x)_\infty\right|\nonumber\\
& \leq \frac{1}{d^n-1} \rho(x,\zetaG) + \left| \frac{1}{d^n-1} \hat{\ell}^{(n)}_{\phi, \infty}(x) - \hat{h}_{\phi, \infty}(x) - \frac{1}{d^n-1} \log_v \delta(x,x)_\infty\right|\label{eq:splitestimaterightside}
\end{align}
Using (\ref{eq:preppiece}), we estimate the second term in (\ref{eq:splitestimaterightside}) as\footnote{Note these are real-valued functions, so we cannot use the ultrametric inequality.}:
\begin{align}
\left| \dfrac{1}{d^n-1} \hat{\ell}^{(n)}_{\phi, \infty}(x) \right.& \left.- \hat{h}_{\phi, \infty}(x) -\dfrac{1}{d^n-1} \log_v(\delta(x  ,x)_{\infty})\right|= \nonumber\\
& = \left| \frac{d^n}{d^n-1}\frac{1}{d^n} \hat{\ell}_{\phi, \infty}^{(n)}(x) - \frac{d^n}{d^n-1} \hat{h}_{\phi, \infty}(x) + \frac{1}{d^n-1} \hat{h}_{\phi, \infty}(x) - \frac{1}{d^n-1} \log_v \delta(x,x)_\infty\right|\nonumber\\
&\leq\left| \dfrac{d^n}{d^n-1} \left(\dfrac{1}{d^n} \hat{\ell}^{(n)}_{\phi, \infty}(x) - \hat{h}_{\phi, v}(x) \right) \right| +\left| \dfrac{1}{d^n-1} \left( \hat{h}_{\phi, \infty}(x) - \log(\delta(x,x)_{\infty})\right)\right|\label{eq:firstestconv}\ .
\end{align}

By the construction of $\hat{h}_{\phi, \infty}(x)$ on $\pberk$ (see \cite{BR}, Section 10.1), there is a constant $C_1= C_1(\phi)$ depending only on $\phi$ so that the first term in term (\ref{eq:firstestconv}) is bounded above:
\begin{align}\label{eq:firstpiece}
\left| \dfrac{d^n}{d^n-1} \left(\dfrac{1}{d^n} \hat{\ell}^{(n)}_{\phi, \infty}(x) - \hat{h}_{\phi, \infty}(x) \right) \right| \leq \dfrac{1}{d-1}\cdot \frac{1}{d^n-1}\cdot C_1\ .
\end{align}  We next rewrite the second term in (\ref{eq:firstestconv}) as
\begin{align}
\left| \dfrac{1}{d^n-1} \left(\hat{h}_{\phi, \infty}(x) - \log_v\delta(x,x)_\infty\right)\right|&\leq \left|\frac{1}{d^n-1} \left(\hat{h}_{\phi, \infty}(x) - \log_v \max(1, [T])_x)\right)\right|\label{eq:secondpiece}\\
& + \left|\frac{1}{d^n-1} \left(\log_v \max(1, [T]_x) - \log_v \delta(x,x)_\infty\right)\right|\label{eq:lastpiece}\ .
\end{align} 

Applying \cite{BR} Equation (10.11) (see also the remark after \cite{BR} Equation (10.7)), the expression on the right side of (\ref{eq:secondpiece}) is no greater than $\frac{1}{(d-1)(d^n-1)}C_1$. For the term appearing in (\ref{eq:lastpiece}), we have that 
\begin{align*}
\left|\frac{1}{d^n-1} \left(|\log_v \max(1, [T]_x) - \log_v \delta(x,x)_\infty\right)\right| &= \frac{1}{d^n-1} \big( \log_v \max(1, |a|, |b|) - \log_v |b|\big)\\
& \leq \frac{1}{d^n-1} \big( 2 \log_v \max(1, |a|, |b|) - \log_v |b|\big)\\
& = \frac{1}{d^n-1} \rho(x, \zetaG)\ ,
\end{align*} where the final equality was established in the course of the proof of Lemma~\ref{lem:precomp}. Inserting these estimates into (\ref{eq:splitestimaterightside}), we find
\begin{align*}
\left| \dfrac{1}{d^n-1} \hat{k}^{(n)}_{\phi, \infty}(x) \right.&\left.- \hat{h}_{\phi, \infty}(x) - \dfrac{1}{d^n-1} \log_v(\delta(x,x)_{\infty})\right|& \\
& \leq \frac{1}{d^n-1} \rho(x, \zetaG) + \frac{1}{d^n-1}\cdot \frac{2C_1}{d-1} + \frac{1}{d^n-1} \rho(x, \zetaG)\\
& \leq \frac{2}{d^n-1} \max\left(\rho(x, \zetaG), \frac{C_1}{d-1}\right)\ .
\end{align*}Letting $C_\phi = \frac{C_1}{d-1}$ gives the asserted bound.
\end{proof}

\noindent\textbf{Remark:} Note that Proposition~\ref{prop:htconvergent} gives an effective, geometrically convergent algorithm for approximating the Berkovich canonical height $\hat{h}_{\phi, \infty}(x)$ by using the convergents $\hat{k}^{(n)}_{\phi}(x)$ instead of the `classical' convergents $\hat{\ell}_{\phi,\infty}(x)$. The advantage of these new convergents is that they require only taking the maximum over the coefficients of $(F^{(n)})^\gamma$, $(G^{(n)})^\gamma $ rather than the supremum of their values on discs.\\

We are now ready to show the convergence of the normalized function $\dfrac{1}{d^{2n}-d^n} \ord\Res_{\phi^n}(x)$ to the function $g_{\mu_{\phi}}(x,x)$. 

\begin{proof}[Proof of Theorem~\ref{thm:fnconvexplicit}]
Let $x=\gamma(\zetaG)$, where $\gamma(z) = bz+a$. Using the decompositions in Table \ref{table:decomp} above we have
\begin{align}
\left|\frac{1}{d^{2n}-d^n}\ord\Res_{\phi^n}(x) - g_{\mu_\phi}(x,x)\right| &\leq \left|\frac{1}{d^{2n}-d^n}\ord\Res(F^{(n)}, G^{(n)}) - \left(\frac{-1}{d(d-1)} \log|\Res(F, G)|\right)\right| \label{eq:firstpart}\\ 
& +\left|\frac{d^{2n}+d^n}{d^{2n}-d^n} \ord\det(\gamma) + \log(\delta(x,x)_{\infty})\right. \label{eq:secondpart}\\
 &  \left.-\frac{2}{d^{n}-1} \min\left(\ord\left((F^{(n)})^{\gamma}\right), \ord\left((G^{(n)})^{\gamma}\right)\right) - 2\hat{h}_{\phi, \infty}(x)\right|\ . \label{eq:thirdpart}
\end{align}

By Lemma \ref{lem:ordresit}, the term (\ref{eq:firstpart}) is identically zero. By Lemma \ref{lem:orddethsia}, the term in (\ref{eq:secondpart}) above is 
\begin{align*}
-\dfrac{2}{d^n-1} \log_v(\delta(x,x)_{\infty}))\ .
\end{align*} The terms in (\ref{eq:secondpart}) and (\ref{eq:thirdpart}) are precisely (twice) the terms bounded in Proposition \ref{prop:htconvergent}, and so we have 
\begin{align}
\left| \dfrac{1}{d^{2n}-d^n} \ord\Res_{\phi^{(n)}}(x) - g_{\mu_{\phi}}(x,x)\right| &\leq \dfrac{4}{d^n-1}\max\left(C_\phi, \rho(x, \zetaG)\right)\ . \label{eq:finalconvest}
\end{align}

This establishes both pointwise convergence on type II points and uniform convergence in the sets $\B_{\rho}(\zetaG, R)$ for fixed $R>0$.
\end{proof}

We note the following corollary to the convergence; the result in fact holds for diagonal Green's functions attached to arbitrary probability measures, as we will see in Lemma~\ref{lem:slopegreensfcn} below.
\begin{cor}\label{cor:gxxprops} The function $g_{\mu_\phi}(\cdot, \cdot)$ is convex up on segments $[x,y]$ in $\hberk$ and is Lipschitz continuous with respect to the path distance metric with Lipschitz constant 1.\end{cor}
\begin{proof}
Let $r_n(x) := \frac{1}{d^{2n}-d^n} \ord\Res_{\phi^n}(x)$; it was shown in \cite{Ru1} Proposition 1.3 that this function is convex up on $\pberk$, and also that the $r_n$ are each Lipschitz continuous with Lipschitz constant $1+\frac{2}{d^n-1}$.

For the convexity of $g_{\mu_\phi}$, fix a segment $[x,y]\subseteq\hberk$. There exists a constant $R>0$ for which $[x,y]\subseteq B_{\rho}(\zetaG, R)$, and so we may assume that $r_n \to g_{\mu_\phi}$ uniformly on $[x,y]$. 

For brevity of notation, let $g(z) = g_{\mu_\phi}(z,z)$. Fix $t\in [x,y]$; we need to show $$\dfrac{g(y)-g(x)}{\rho(y,x)} \geq \dfrac{g(t)-g(x)}{\rho(t,x)}\ .$$ By Theorem \ref{thm:fnconv}, we can choose $n$ large enough so that $\left|\frac{g(y)-r_n(y)}{\rho(y,x)}\right| < \frac{\epsilon}{4}\ ,$ $\left|\frac{r_n(x) - g(x)}{\rho(y,x)}\right| < \frac{\epsilon}{4}\ ,$ and $\left|\frac{g(t)-r_n(t)}{\rho(t,x)}\right| < \frac{\epsilon}{4}\ .$ We have 
\begin{align*}
\frac{g(y)-g(x)}{\rho(y,x)} &\ =\  \frac{g(y) - r_n(y) + r_n(y) - r_n(x) + r_n(x) - g(x)}{\rho(y,x)}\\
& \ \geq \ -\frac{\epsilon}{4} -\frac{\epsilon}{4} + \frac{r_n(y) - r_n(x) }{\rho(y,x)}\\
& \ \geq \ -\frac{\epsilon}{2} + \frac{r_n(t) - r_n(x)}{\rho(t,x)}\\
& \ = \ -\frac{\epsilon}{2} + \frac{r_n(t) - g(t) + g(t) - g(x)+g(x) - r_n(x)}{\rho(t,x)}\\
& \ \geq \ -\frac{\epsilon}{2} -\frac{\epsilon}{4} - \frac{\epsilon}{4} + \frac{g(t) - g(x)}{\rho(t,x)}\\
& \ = \ -\epsilon + \frac{g(t) - g(x)}{\rho(t,x)}\ .
\end{align*} Since $\epsilon>0$ was arbitrary, we conclude that $g(t)$ is convex up on $[x,y]$.

To see that $g_{\mu_\phi}(x,x)$ is Lipschitz continuous, fix $x,y\in \hberk$, and let $0<\epsilon<\frac{\rho(x,y)}{2}$. Choose $n$ sufficiently large so that $$|g(x)-r_n(x)| < \epsilon/2\ \textrm{ and }\ |g(y)-r_n(y)| < \epsilon/2\ .$$ Since $|r_n(x) - r_n(y)|\leq 1+\frac{2}{d^n-1}$, we have \begin{align*} |g(x)-g(y)| &= |g(x)-r_n(x)+r_n(x)-r_n(y)+r_n(y)-g(y)|\\ & \leq \left(1+\frac{2}{d^n-1}\right)\rho(x,y) + \epsilon\ .\end{align*} Letting $\epsilon \to 0$ and $n\to \infty$ gives the desired result.
\end{proof}

\section{Weak Convergence of the Crucial Measures} \label{sect:lapconv}
We now apply the results of the previous section to show the weak convergence of the measures $\nu_{\phi^n}$ to the canonical measure $\mu_{\phi}$. The proof will come from the following more explicit theorem

\begin{thm}\label{thm:wkconvexplicit}
If $\Gamma\subseteq\hberk$ is a finite connected subtree and $f$ is a continuous piecewise affine map on $\Gamma$, then there exist a constant $C_\phi>0$ depending only on $\phi$, and constants $R_\Gamma, D_\Gamma>0$ depending only on $\Gamma$ so that 
$$\left|\int_{\pberk} f d(\mu_{\phi}-\nu_{\phi^n})\right| \leq \dfrac{4}{d^n-1} \left(\max\left(C_{\phi}, R_\Gamma\right)\cdot  |\Delta|(f) + \max_{\Gamma} |f| \cdot D_{\Gamma}\right)\ .$$
\end{thm}
We prove Theorem~\ref{thm:wkconvexplicit} in Section \ref{sect:lapweakconv} below, and after we give the proof of Theorem~\ref{thm:wkconv}. Before doing either of these, we briefly outline the remainder of this section. In Section~\ref{subsect:slopesagain} we expand on the slope formulae for $\ord\Res_{\phi}(x)$ developed in \cite{Ru2} to include slopes at arbitrary points in $\hberk$. It turns out that the more convenient function to study is \begin{align}\label{eq:modifiedfn} f_n(x) = \dfrac{1}{d^{2n}-d^n} \ord\Res_{\phi^n}(x) +\log(\delta(x,x)_{\infty})\ .\end{align} Working with this function will give cleaner slope functions when computing the Berkovich Laplacian. We also note that \begin{align}\label{eq:desiredprop} \dfrac{1}{d^{2n}-d^n}\ord\Res_{\phi^n}(x) - g_{\mu_{\phi}}(x,x) = f_n(x) - 2\hat{h}_{\phi,\infty}(x) + \log_v|\Res(\Phi)|^{-1/(d(d-1))}\ .\end{align} 

Our strategy in proving Theorem~\ref{thm:wkconvexplicit} is as follows: let $\Gamma_{\textrm{FR, n}}$ denote the tree spanned by the type I $n$-periodic points and type II repelling $n$-periodic points, and let $\Gamma_{\widehat{\textrm{FR}}, n}$ denote the truncation of this tree obtained by removing segments $[\alpha, Q_0)$, where $\alpha$ is a type I $n$-periodic point and $Q_0$ is a type II point chosen so that $\ordRes_{\phi^n}(\cdot)$ has constant slope along $[\alpha, Q_0)$. We compute the Laplacian of $f_n(x)$ on arbitrary finite connected subtrees $\Gamma\subseteq \hberk$, first by joining such a subtree to $\Gamma_{\widehat{FR},n}$ and then computing the (retraction of the) Laplacian on the larger subtree. We prove weak convergence for $\CPA$ test functions on an (arbitrary) fixed subtree $\Gamma\subseteq\hberk$, and then using an approximation theorem for continuous functions on $\pberk$ extend this to show weak convergence in general. 

\subsection{Slope Formulae Revisited}\label{subsect:slopesagain}

Here we compute the slope of the functions $f_n(x)$ on finite connected subtrees $\Gamma\in\hberk$ which share at most one point in common with the corresponding $\Gamma_{\textrm{FR}, n}$; in the following section these results will be used to give explicit formulae for $\Delta_{\Gamma}(f_n)$ for such $\Gamma$. The parallel result for subtrees $\Gamma\subseteq\Gamma_{\widehat{FR},n}$ is found in \cite{Ru2} Corollary 6.5, which will be discussed in the next section. 

\begin{lemma}\label{lem1}
Let $\Gamma\subseteq\hberk$ be any finite tree. Let $\mu_{Br, \Gamma}$ be the branching measure attached to this tree. Then $$ \Delta_\Gamma(\log(\hsia{\cdot}{\cdot}{\infty}) ) = -2\mu_{Br, \Gamma} + 2\delta_{r_{\Gamma}(\infty)}$$ where $r_{\Gamma}(\infty)$ is the retraction of $\infty$ to $\Gamma$.
\end{lemma}

\begin{proof}
This is a straightforward computation. Let $w=r_{\Gamma}(\infty)$. Note that $\log(\hsia{\cdot}{\cdot}{\infty})$ is the arclength parameterization for a segment of $\Gamma$. Fix $P\in \Gamma \setminus \{w\}$, and let $\vv_w\in T_P$ be the unique direction towards $w$. We have
\begin{align*}
\sum_{\vv\in T_P \Gamma}\del_{\vv}(\log(\hsia{\cdot}{\cdot}{\infty}))(P)=& \sum_{\vv\neq \vv_w} \del_{\vv}(\log(\hsia{\cdot}{\cdot}{\infty}))(P) + \del_{\vv_w}(\log(\hsia{\cdot}{\cdot}{\infty}))(P)\\
=& \left(\sum_{\vv\neq \vv_w} -1\right) +1\\
=& (v_\Gamma(P) - 1)(-1) + 1\\
=& (2-v_\Gamma(P))\ .
\end{align*}
For $P=w$:
\begin{align*}
\sum_{\vv\in T_w\Gamma} \del_{\vv}((\log(\hsia{\cdot}{\cdot}{\infty}))(P) =& \sum_{\vv\in T_w\Gamma} -1\\
=& -v_\Gamma(w)\\
=& (2-v_\Gamma(w)) -2\ .
\end{align*}
Thus 
\begin{align*}
\Delta_{\Gamma}(\log(\hsia{\cdot}{\cdot}{\infty})) =& \sum_{P\in \Gamma} (v_\Gamma(p)-2)\delta_P + 2\delta_w \\
=& -2\mu_{Br, \Gamma} + 2\delta_{r_{\Gamma}(\infty)}\ .
\end{align*}
\end{proof}

In the next two lemmas, we compute slopes of the function $f_n$ introduced above. For this, we will rely on several Identification Lemmas of Rumely, given in \cite{Ru2} Lemmas 2.1, 2.2, and 4.5. These formulas depend on several types of multiplicities associated to a rational map, which we recall now. We first give definitions of these multiplicities at $P=\zetaG$ under the assumption $\phi(\zetaG) = \zetaG$; the case for general type II points can be recovered from this via pre-and-post conjugation of $\phi$. We also assume that we have fixed a coordinate system so that $T_{\zetaG} \simeq \PP^1(k)$, and will write $\vv_{\tilde{a}}$ for the direciton corresponding to $\tilde{a}\in \PP^1(k)$. Here, the tilde is important, and $\vv_{\tilde{a}}$ should not be confused with $\vv_a\in T_{\zetaG}$, where $a\in \pberk$; by $\vv_a$ we mean the direction for which $a\in B_{\zetaG}(\vv_a)^-$.

\begin{enumerate}
\item The {\em fixed point multiplicity} in the direction $\vv\in T_{\zetaG}$ is given $$\#F_\phi(\zetaG, \vv) = \textrm{ the number of type I fixed points in } B_{\zetaG}(\vv)^-\ .$$
\item The {\em reduced fixed point multiplicity} in the direction $\vv_{\tilde{a}}\in T_{\zetaG}$ is given $$\#\tilde{F}_\phi(\zetaG, \vv_{\tilde{a}})  = \textrm{ multiplicity of  }\tilde{a}\textrm{ as a fixed point of the reduction of }\phi\textrm{ at }P\ .$$ 
\item The {\em surplus multiplicity} in the direction $\vv_{\tilde{a}}\in T_{\zetaG}$ can be given as follows: Let $[F, G]$ be a normalized lift of $\phi$, and write $\tilde{F} = \tilde{A}\cdot \tilde{F}_0, \tilde{G} = \tilde{A} \cdot \tilde{G}_0$, where $\tilde{A} = \gcd(\tilde{F}, \tilde{G})$. Then the surplus multiplicity of $\phi$ at $\zetaG$ in the direction $\vv_{\tilde{a}}\in T_{\zetaG}$ is $$s_\phi(\zetaG, \vv_{\tilde{a}}) = \textrm{multiplicity of } \tilde{a} \textrm{ as a root of } \tilde{A}\ .$$ There is also a geometric interpretation of the surplus multiplicity (see \cite{XF} Proposition 3.10). The surplus multiplicity is closely related to the ramification of the map $\phi$ (see \cite{XF}). \\
\end{enumerate}

Now let $\Gamma$ be a finite, connected subtree of $\hberk$ that intersects $\Gamma_{FR,n}$ in at most one point. For fixed $n$, let $w_n$ denote the point of $\Gamma$ that is nearest to $\Gamma_{FR,n}$.

\begin{lemma}\label{lem:slopenotw}
Fix a type II point $P\in \Gamma\setminus\{w_n\}$, and let $i(P) = \left\{\begin{matrix} 0\ , & P \neq r_\Gamma(\infty)\\ -2\ , & P = r_\Gamma(\infty)\end{matrix}\right.\ .$ Then 

$$\sum_{\vv\in T_P \Gamma}\del_{\vv}(f_n)(P) = i(P)+ 
\begin{cases} 

\frac{2}{d^n-1}(v_\Gamma(P)-2), & \text{if }\phi^n(P)\neq P, \\
\frac{2}{d^n-1}(v_\Gamma(P)-1), & \text{if }\phi^n(P) = P  \text{ and }P\text{ is not id-indifferent for }\phi^n,\\
0, & \text{if }\phi^n(P) = P \text{ and }P\text{ is id-indifferent for }\phi^n.\\
\end{cases}
$$
\end{lemma}
\begin{proof}

We begin with the case of $\phi^n(P) \neq P$. Here we use the formula from Proposition 5.4 in \cite{Ru2}, together with the calculations from Lemma~\ref{lem1}. By a slight abuse of notation, let $\vv_w := \vv_{w_n}$ be the direction at $P$ pointing towards $w_n$. Note that $\#F_{\phi^n}(P,\vv) = 0$ for any $\vv\neq \vv_w$, and $\#F_{\phi^n}(P,\vv_w) = d^n+1$. In particular, applying \cite{Ru2} Proposition 5.4 we find
\[
\del_{\vv} (\ordRes_{\phi^n})(P) = \left\{\begin{matrix} d^{2n} + d^n\ , & \vv \neq \vv_w\\ -(d^{2n} + d^n) \ , & \vv = \vv_w\end{matrix}\right.\ .
\] Summing over all directions gives
\[
\sum_{\vv\in T_P \Gamma} \del_{\vv} (\ordRes_{\phi^n})(P) = (d^{2n} + d^n) (v_\Gamma(P) - 2)\ .
\] In the proof of Lemma~\ref{lem1}, we saw
\[
\sum_{\vv\in T_P\Gamma} \del_{\vv} (\log(\delta(\cdot, \cdot)_\infty))(P) = \left\{\begin{matrix} 2-v_\Gamma(P)\ , & P\neq r_\Gamma(\infty)\\ -v_\Gamma(P)\ , & P = r_\Gamma(\infty)\end{matrix}\right\} = (2-v_\Gamma(P)) + i(P)\ .
\] Therefore, we find that 
\[
\sum_{\vv\in T_P\Gamma} \del_{\vv}(f_n)(P) = \frac{2}{d^n-1} (v_\Gamma(P) - 2)+ i(P)\ .
\]

In the case $\phi^n(P)=P$, we need to separate into the cases where $P$ is id-indifferent for $\phi^n$ and where it is not. 

If $P$ is not id-indifferent for $\phi^n$, then the First Identification Lemma (\cite{Ru2}, Lemma 2.1) implies that for all directions $\vv\in T_P \Gamma \setminus \{\vv_w\}$, we have $s_{\phi^n}(P,\vv) = 0$ and $(\phi^n)_*(\vv)\neq \vv$. For $\vv=\vv_w$, we claim that $s_{\phi^n}(P,\vv) = d^n-1$ and $(\phi^n)_*(\vv) = \vv$.  To see this, we note that $F_{\phi^n}(P, \vv_w) = d^n +1$, as all of the fixed points lie in the direction towards $w$. The first identification lemma then gives that $$d^n+1 = s_{\phi^n}(P, \vv_w) + \tilde{F}_{\phi^n}(P, \vv_w)\ .$$ Since $s_{\phi^n}(P, \vv_w) \leq d^n$, this forces $\tilde{F}_{\phi^n}(P, \vv_w) \geq 1$, i.e. $(\phi^n)_* \vv_w = \vv_w$. Now note that $P$ is additively indifferent for $\phi^n$: by assumption it is not id-indifferent, and it cannot be repelling, as it does not lie in $\Gamma_{\textrm{FR}, n}$. If it were multiplicatively indifferent, then it would necessarily fix two directions, since $\widetilde{\phi^n}(z) = \tilde{\lambda}z$ fixes two points of $\PP^1(k)$; but the First Identification Lemma (\cite{Ru2} Lemma 2.1) would then imply that there are two directions containing fixed points, which is a contradiction. So $P$ is addtivively indifferent for $\phi^n$. Since the linear map $z\mapsto z+\tilde{\lambda}$ has a unique degree 2 fixed point, we conclude that $\tilde{F}_{\phi^n}(P, \vv_w) = 2$, hence $s_{\phi^n}(P, \vv_w) = d^n-1$. Therefore, applying \cite{Ru2} Proposition 5.2, we find
\[
\del_{\vv}(\ordRes_{\phi^n})(P) = \left\{\begin{matrix} d^{2n}+d^n \ , & \vv \neq \vv_w\\ 
-d^{2n} + d^n\ , & \vv = \vv_w\ .\end{matrix}\right.\ .
\] Summing over all directions gives
\[
\frac{1}{d^{2n}-d^n} \sum_{\vv\in T_P \Gamma} \del_{\vv}(\ordRes_{\phi^n})(P) = \frac{d^{2n}+d^n}{d^{2n}-d^n} (v_\Gamma(P) - 1) - 1\ .
\] Accounting also for the contribution of $\log(\delta(\cdot, \cdot)_\infty)$, we find
\[
\sum_{\vv\in T_P\Gamma} \del_{\vv}(f_n)(P) = \frac{2}{d^n-1} (v_\Gamma(P) - 1) + i(P)
\]

The proof when $P$ is id-indifferent is similar. We again claim that $s_{\phi^n}(P, \vv_w) = d^n-1$: to see this, we apply the Third Identification Lemma (\cite{Ru2} Lemma 4.5), which implies that $\vv_w$ is the only direction with $s_{\phi^n}(P, \vv_w)>0$. Since $d^n = \deg_{\phi^n}(P) + \sum_{\vv\in T_P} s_{\phi^n}(P, \vv)$ (see \cite{XF} Equation 3.1), this implies that $d^n-1 = s_{\phi^n}(P, \vv_w)$. In particular, by \cite{Ru2} Proposition 5.3, we find
\[
\del_{\vv} (\ordRes_{\phi^n})(P) = \left\{\begin{matrix} d^{2n}-d^n\ , &  \vv\neq \vv_w\\
-d^{2n}+d^n \ , & \vv = \vv_w\end{matrix}\right.\ .
\] Summing over all places gives
\[
\frac{1}{d^{2n}-d^n} \sum_{\vv\in T_P\Gamma} \del_{\vv}(\ordRes_{\phi^n})(P) = (v_\Gamma(P) - 2)
\] Adding the contribution from the $\log(\delta(\cdot, \cdot)_\infty)$ term gives
\[
\sum_{\vv\in T_P\Gamma} \del_{\vv}(f_n)(P) = i(P)\ ,
\]
\end{proof}

We are left to consider the case of $P=w_n$, where we recall that $w_n$ is the point of $\Gamma$ which is nearest to $\Gamma_{\textrm{FR}, n}$. If $P=w_n$,then  $s_{\phi^n}(w_n,\vv)$ and $\#F_{\phi^n}(w_n,\vv)$ are zero for all directions $\vv$ pointing into $\Gamma$; to see this, note that these are precisely the directions which point \emph{away} from $\Gamma_{\textrm{FR}, n}$, hence $F_{\phi^n}(w_n, \vv) = 0$. If $w$ is fixed by $\phi^n$, but not id-indifferent, then the First Identification Lemma (\cite{Ru2} Lemma 2.1) gives that $0=s_{\phi^n}(w_n, \vv) + \tilde{F}_{\phi^n}(w_n, \vv)$, hence $s_{\phi^n}(w_n, \vv) = 0$. If $w$ is id-indifferent, then the Third Identification Lemma (\cite{Ru2} Lemma 4.5) implies that $s_{\phi^n}(w_n, \vv) = 0$ since there are no type I fixed points in $B_{w_n}(\vv)^-$. Likewise, if $w_n$ is moved by $\phi^n$, then the Second Identification Lemma (\cite{Ru2} Lemma 2.2) gives that $s_{\phi^n}(w_n, \vv) = 0$. 

\begin{lemma} \label{lem:slopew}
Suppose $P=w_n \in \Gamma$ is a type II point, and let $i(P) = i(w_n)$ be as in the statement of Lemma~\ref{lem:slopenotw}. Then
$$
\sum_{\vv\in T_{w_n}\Gamma} \del_{\vv} (f_n)(P)  = 2+ i(P) + \begin{cases} 0, & w_n \mbox{ is an id-indifferent fixed point,}\\
\dfrac{2}{d^n-1} v_\Gamma(w_n), & \mbox{ otherwise.}

\end{cases}
$$
\end{lemma}

\begin{proof}
Here again the proof splits into three cases. If $w_n$ is not fixed, then by \cite{Ru2} Proposition 5.4,
\[
\sum_{\vv\in T_{w_n} \Gamma} \del_{\vv}( \ord\Res_{\phi^n})(P) = (d^{2n} + d^n)v_\Gamma(w_n)
\]  for all $\vv\in T_{w_n}(\Gamma)$.  Again taking into account the contribution from $\log(\delta(\cdot, \cdot)_\infty)$, we find
\begin{align*}
\sum_{\vv\in T_{w_n}} \del_{\vv}(f_n)(P) &=  \frac{d^{2n}+d^n}{d^{2n}-d^n} v_\Gamma(w_n) + 2- v_\Gamma(w_n) + i(w_n)\\
& = \frac{2}{d^n-1} v_\Gamma(w_n) + 2+i(w_n)\ .
\end{align*}

If $w_n$ is fixed by $\phi^{n}$ but is not id-indifferent, we argued above that $s_{\phi^n}(w_n, \vv) = 0$, hence $(\phi^n)_* \vv \neq \vv$ for any $\vv$ pointing into $\Gamma$ (this is again the First Identification Lemma, \cite{Ru2} Lemma 2.1). Therefore, applying \cite{Ru2} Proposition 5.2, we find
\[
\del_{\vv}(\ordRes_{\phi^n})(P) = d^{2n} + d^n
\] for all $\vv\in T_{w_n}(\Gamma)$. As in the previous case, summing over all directions and taking into account the contribution from $\log(\delta(\cdot, \cdot)_\infty)$ gives
\[
\sum_{\vv\in T_{w_n}\Gamma} \del_{\vv}(f_n)(P) =  \frac{2}{d^n-1} v_\Gamma(w_n) + 2+i(w_n)\ .
\]

Finally, if $w_n$ is fixed by $\phi^n$ and is id-indifferent, applying \cite{Ru2} Proposition 5.3 we find
\[
\del_{\vv}( \ordRes_{\phi^n})(P) = d^{2n}-d^n
\]  for all $\vv\in T_{w_n}(\Gamma)$. Therefore, summing over all places and taking into account the contribution from $\log(\delta(\cdot, \cdot)_\infty)$ gives
\[
\sum_{\vv\in T_{w_n}\Gamma} \del_{\vv}(f_n)(P)  = v_\Gamma(w_n) + 2 - v_\Gamma(w_n) + i(w_n) = 2+i(w_n)\ .
\]
\end{proof}

\noindent\textbf{Remark:} From now until Section~\ref{sect:exampleswkconv}, we will assume that $\infty$ is a fixed point of $\phi$, which can always be arranged by conjugating $\phi$ with an appropriate $\gamma\in \GL_2(K)$. The equivariance of the measures $\mu_\phi$ and $\nu_\phi$ ensures that Theorem~\ref{thm:wkconvexplicit} is valid if we replace $\phi$ by such a conjugate. 

\subsection{Applications To Laplacians}\label{sect:applap}
In this section we use the slope formulae described above to relate the Laplacian of $f_n$ to both the crucial measure and the canonical measure; these relations will be key in the estimates that give weak convergence. 

Recall that $\Gamma_{FR,n}$ is the tree in $\pberk$ spanned by the classical $n$-periodic points and the type II repelling $n$-periodic points; in \cite{Ru2} it was shown that this is spanned by finitely many points. Fix $n$, and fix also a finite tree $\Gamma\subseteq\hberk$. Take $R_{\Gamma}>0$ sufficiently large so that $\Gamma \subseteq B_{\rho}(\zetaG, R_\Gamma)$. 

We work with a truncated version of $\Gamma_{\textrm{FR},n}$, which we denote by $\Gamma_{\widehat{FR},n}$. It is constructed as follows: for each classical fixed point $\alpha_i$ of $\phi^n$, choose a point $Q_i\in \hberk$ sufficiently near to $\alpha_i$ so that $[Q_i, \alpha_i]$ contains no branch points of $\Gamma_{FR,n}$, and that the slope of $\ord\Res_{\phi^n}(\cdot)$ is constant on this segment (see \cite{Ru2}). We further extend these branches if necessary to ensure that none of the endpoints of $\Gamma_{\widehat{\textrm{FR}}, n}$ lie in $B_{\rho}(\zetaG, R_\Gamma)$; because of this last condition, our $\Gamma_{\widehat{\textrm{FR}}, n}$ differs slightly from the truncated trees used in \cite{Ru2}.

Our goal is to compute the Laplacian of $f_n$ on $\Gamma$, which we will do by first computing the Laplacian of $f_n$ on a larger tree $\Gamma^{(n)}$, and then retracting to $\Gamma$. If $\Gamma\cap \Gamma_{\widehat{FR}, n} \neq \emptyset$, let $\Gamma^{(n)} = \Gamma \cup \Gamma_{\widehat{FR}, n}$. Otherwise, letting $x_n\in \Gamma_{\widehat{FR}, n}$ denote the point in $\Gamma_{\widehat{FR}, n}$ lying closest to $\Gamma$, and $w_n\in \Gamma$ the point lying closest to $\Gamma_{\widehat{FR}, n}$, we define $\Gamma^{(n)} = \Gamma \cup [w_n, x_n]$. In either case, we let $\Gamma^{(n)}_0 = \Gamma^{(n)}\cap \Gamma_{\widehat{FR}, n}$, and for $i= 1, 2, ..., N=N(n, \Gamma)$, we let $\Gamma^{(n)}_i$ denote the \emph{closures} of the various components of $\Gamma^{(n)}\setminus \Gamma^{(n)}_0$.\\

\noindent \textbf{Remark}: Each $\Gamma^{(n)}_i$ intersects $\Gamma_{\widehat{\textrm{FR}}, n}$ in exactly one point, which is necessarily type II: if the intersection is along an edge of $\Gamma_{\widehat{\textrm{FR}},n}$, then the point of intersection necessarily has at least three directions (at least two into $\Gamma_{\widehat{\textrm{FR}}, n}$ and one into $\Gamma^{(n)}_i$. Hence it is a type II point. Otherwise, the intersection is at an endpoint of $\Gamma_{\widehat{\textrm{FR}}, n}$; by our choice of $R_\Gamma$ it cannot be a point $Q_0$ from the truncation, and the only other endpoints are Berkovich repelling points (\cite{Ru2} Proposition 4.3), which are necessarily type II. \\

We begin with a lemma that shows that, in the case $\Gamma\cap \Gamma_{\widehat{FR}, n} \neq \emptyset$, the number $N(n, \Gamma)$ is uniformly bounded in terms of $\Gamma$:

\begin{lemma}\label{lem:treelemma}
 There is a constant $K(\Gamma)$ such that for any finite tree $\Gamma\subseteq\hberk$ having exactly $s$ edges, and any connected subtree $\Gamma_0\subseteq\Gamma$, the number of connected components of $\Gamma\setminus \Gamma_0$ is bounded by $K(\Gamma) = 2s$.
\end{lemma}

\begin{proof}
we proceed by induction on the number of edges of $\Gamma$. If $\Gamma$ has one edge, then $\Gamma$ is an interval and any connected subset is again an interval, so by removing a subinterval we form at most 2 connected components. Hence $K(\Gamma) = 2$, which establishes the claim when $s=1$.

Now let $\Gamma$ be a tree with $s>1$ edges, and let $\Gamma_0\subseteq \Gamma$ be a connected subtree. Let $e$ be an edge of $\Gamma$ such that $\Gamma\setminus e$ is connected, and let $\hat{\Gamma}$ be the tree formed by removing $e$ from $\Gamma$. Similarly, let $\hat{\Gamma_0}$ be the tree formed by removing from $\Gamma_0$ any part that intersects $e$. By induction, $\hat{\Gamma}\setminus \hat{\Gamma_0}$ has at most $K(\hat{\Gamma}\setminus\hat{\Gamma_0}) = 2(s-1)$ components. If we consider the edge $e$ alone, then it will have at most two components in $e\setminus (e\cap \Gamma_0)$; thus by adding $e$ back into $\hat{\Gamma}$ and $\hat{\Gamma_0}$ we will add at most two components to $\hat{\Gamma} \setminus \hat{\Gamma_0}$ (in general, only one, unless $\Gamma_0\subseteq e$), thus there are at most $2(s-2) + 2 = 2s$ components of $\Gamma\setminus \Gamma_0$, i.e $K(\Gamma) = 2s$.
\end{proof}
In particular, the preceeding lemma shows that $N(n, \Gamma)$ is bounded above by $2s$, where $s$ is the number of edges of $\Gamma$. 

In the following two subsections, we compute the Laplacian of $f_n$ on each of the $\Gamma^{(n)}_i$ for $i\geq 0$; the results are combined to give the Laplacian on $\Gamma^{(n)}$ (Lemma \ref{lem:lapgamman}) and finally the Laplacian of $\frac{1}{d^{2n}-d^n}\ord\Res_{\phi^n}(x) - g_{\mu_\phi}(x,x)$ on $\Gamma$ (Proposition \ref{prop:lapgamma}).

\subsubsection{The Laplacian of $f_n$ on $\Gamma^{(n)}_0$}
Recall that $\Gamma^{(n)}_0$ is either the intersection of $\Gamma$ with $\Gamma_{\widehat{\textrm{FR}}, n}$, or it consists of the unique point in $\Gamma_{\widehat{FR}, n}$ lying closest to $\Gamma$; in either case, we derive a formula for $\Delta_{\Gamma^{(n)}_0}(f_n)$, which is given below in Proposition \ref{prop:lapgamma0}. As a first step in this direction we recall:

\begin{cor}\label{cor:defncrucialmeasures}$($Rumely, \cite{Ru2}$)$
Let $\phi(z)\in K(z)$ have degree $d\geq 2$. Then $$\Delta_{\Gamma_{\widehat{FR}}} (\ord\Res_{\phi}(\cdot)) = 2(d^2-d) (\mu_{\Gamma_{\widehat{FR}}, \textrm{Br}} - \nu_{\phi})\ .$$
\end{cor}

\noindent As we mentioned above, the truncation used by Rumely is slightly different than ours: we also required that the truncated ends of the tree $\Gamma_{\widehat{FR}}$ lie outside $B_\rho(\zetaG, R_\Gamma)$. The result given in Corollary~\ref{cor:defncrucialmeasures} is the same with our definition of $\Gamma_{\widehat{FR}}$. 

Corollary~\ref{cor:defncrucialmeasures} is easily generalized to higher iterates, and combined with Lemma \ref{lem1} above we obtain

\begin{lemma}\label{lem:lapGammaFRn}
If $\Gamma_{\widehat{FR},n}$ is as defined above, then $$\Delta_{\Gamma_{\widehat{FR},n}}\left(\dfrac{1}{d^{2n}-d^n} \ord\Res_{\phi^n}(\cdot) + \log(\delta(\cdot, \cdot)_{\infty})\right) =  - 2 \nu_{\phi^n} + 2\delta_{r_{\Gamma_{\widehat{FR},n}}(\infty)\ .}$$
\end{lemma}
\begin{proof}
This is a straightforward computation: by Lemma~\ref{lem1} and Corollary~\ref{cor:defncrucialmeasures}
\begin{align*}
\Delta_{\Gamma_{\widehat{FR},n}}\left(\dfrac{1}{d^{2n}-d^n} \ord\Res_{\phi^n}(\cdot) +\log(\delta(\cdot, \cdot)_\infty)\right) & = \left(2 \mu_{\Gamma_{\widehat{FR},n}, Br} - 2 \nu_{\phi^n}\right) + \left(-2 \mu_{\Gamma_{\widehat{FR},n}, Br} + 2\delta_{r_{\Gamma_{\widehat{FR},n}}(\infty)}\right)\\
&= - 2 \nu_{\phi^n} + 2\delta_{r_{\Gamma_{\widehat{FR},n}}(\infty)}\ .
\end{align*}

\end{proof}

Finally we apply this to the Laplacian on $\Gamma^{(n)}_0$ by taking retractions: 

\begin{prop}\label{prop:lapgamma0} The Laplacian of $f_n$ on $\Gamma^{(n)}_0$ is given
$$\Delta_{\Gamma^{(n)}_0}(f_n) = -2 (r_{\Gamma_{\widehat{FR}, n}, \Gamma^{(n)}_0})_* \nu_{\phi^n} + 2\delta_{r_{\Gamma^{(n)}_0}(\infty)}\ .$$
\end{prop}

\begin{proof}
Since $\Gamma^{(n)}_0\subseteq \Gamma_{\widehat{FR}, n}$, we have $$\Delta_{\Gamma^{(n)}_0}(f_n) = (r_{\Gamma_{\widehat{FR}, n}, \Gamma^{(n)}_0})_* \Delta_{\Gamma_{\widehat{FR}, n}}(f_n)\ .$$ The result now holds by explicitly retracting the expression appearing in Lemma~\ref{lem:lapGammaFRn}.
\end{proof}

\subsubsection{The Laplacian of $f_n$ on $\Gamma^{(n)}_i$ for $i\geq 1$}

We now aim to compute $\Delta_{\Gamma}(f_n)$ on the trees $\Gamma^{(n)}_i$ for $i\geq 1$; these subtrees share exactly one point in common with $\Gamma_{\widehat{FR},n}$, hence we can apply the results of Section \ref{subsect:slopesagain}. We begin with the following:

\begin{prop}\label{prop:laponept}
Fix $i\geq 1$, and for brevity write $\Pi= \Gamma^{(n)}_i$. Let $w^{(n)}_i$ denote the point in $\Pi\cap \Gamma_{\widehat{FR},n}$. We have 
\begin{enumerate} 
\item Every fixed point of $\phi$ in $\Pi$ that is not id-indifferent is additively indifferent. If $A_{\Pi}$ is the number of fixed points of $\phi$ in $\Pi$, and $E_{\Pi}$ is the number of edges of $\Pi$, then $$A_{\Pi} \leq E_{\Pi}\ .$$

\item The Laplacian of $f_n$ on $\Pi$ is $$\Delta_{\Pi} (f_n) = \dfrac{2}{d^n-1}\left( \Theta_{\Pi} + \Omega_{n,i} \delta_{w^{(n)}_i}\right)\ ,$$ where $$\Theta_{\Pi} = -\sum_{\substack{Q\in \Pi\setminus\{w^{(n)}_i\},\\ \phi^n(Q) = Q\\ Q \text{ is not id-indifferent}}} (v_\Pi(Q)-1)\delta_Q- \sum_{\substack{ P\in \Pi\setminus\{w^{(n)}_i\},\\ \phi^{n}(P)\neq P}} (v_\Pi(P)-2)\delta_P\ , $$ and $$\Omega_{n,i} = \begin{cases} v_\Pi(w^{(n)}_i), & \textrm{if } w^{(n)}_i \text{ is not id-indifferent,} \\ 0, & \textrm{if } w^{(n)}_i \text{ is id-indifferent.} \end{cases}\ .$$ are measures depending only on $\Pi$, $\phi$, and $n$.

\item If $B_\Pi = \sum_{P\in \Pi} |v_\Pi(P)-2| + (E_\Pi+1) \cdot \max_{P\in \Pi} v_\Pi(P)$, then 

$$|\Delta|_{\Pi} (f_n) \leq \dfrac{2}{d^n-1} B_{\Pi}\ .$$
\end{enumerate}
\end{prop}

\begin{proof}

For (1), note that by \cite{Ru2}, Corollary 10.6 and the definition of $\Gamma_{FR,n}$, all of the multiplicatively indifferent and repelling fixed points of $\phi^n$ lie in $\Gamma_{FR,n}$; hence any fixed point $P$ of $\phi^n$ in $\Pi$ that is not id-indifferent must be additively indifferent. Let $U_{\textrm{id}}^{(n)}$ denote the locus of id-indifference for $\phi^n$. 

If $P$ is an additively indifferent fixed point of $\phi^n$, it has a unique fixed direction $\vv_*$ with multiplier 1. By the Second Persistence Lemma (\cite{Ru2}, Lemma 9.5), it follows that there is a segment $(P, P_0)$ in $B_P(\vv_*)^-$ on which $\phi^n$ is id-indifferent. Hence $P$ sits on the boundary of $U_{\textrm{id}}^{(n)}$.

Moreover, by Corollary 10.2 of \cite{Ru2}, the closure of the component of the locus of id-indifference which $P$ bounds must contain at least two classical fixed points. Hence there can be no other additively indifferent fixed points in the segment $(P, w^{(n)}_i)$. The number of additively indifferent fixed points in $\Pi$ is therefore bounded by the number of edges $E_\Pi$ in $\Pi$. This completes (1).

To prove (2), we first note that $w^{(n)}_i = r_{\Gamma_i^{(n)}}(\infty)$; this follows from our assumption that $\infty$ is a type I fixed point of $\phi$ (hence it lies in each $\Gamma_{FR, n}$) together with our truncation of $\Gamma_{FR,n}$ (which extended beyond the ball $B_\rho(\zetaG, R_\Gamma)$ that contained $\Gamma$.  Consequently, applying the slope formula appearing in Lemma~\ref{lem:slopenotw}, there is no contribution from $i(P)$, while in the formula in Lemma~\ref{lem:slopew}, we see that $2+i(P) = 0$. The expression for the Laplacian in (2) now follows by a straightforward application of the definition of the Laplacian on a finite subgraph, along with the slope formulae derived in Lemmas \ref{lem:slopenotw} and \ref{lem:slopew}.

 For (3), the contribution from each fixed point $P$ to $\left|\Delta_\Pi\right|(f_n)$ is at most $v_\Pi(P)-1$, and by parts (1) and (2) above there are at most $E_\Pi$ additively indifferent fixed points in $\Pi$. Therefore the first sum in $\Theta_\Pi$ is bounded by $E_\Pi\cdot \max_{P\in \Pi} (v_\Pi(P) -1) \leq E_\Pi\cdot \max_{P\in \Pi} v_\Pi(P)$. The second sum in $\Theta_\Pi$ is evidently bounded above by $\sum_{P\in \Pi} |v_\Pi(p)-2|$. This, together with the $v_\Pi(P)$ term that bounds $\Omega_n \delta_{w^{(n)}_i}$, gives the result.
\end{proof}

Recall that $ N(n, \Gamma)$ is the number of components of $\Gamma^{(n)}\setminus \Gamma_0^{(n)}$. Applying the previous proposition to each of the trees $\Gamma^{(n)}_i$ ($i\geq 1$), we have 

\begin{lemma}\label{lem:lapgammak}
Let $\Theta_{\Gamma_{i}^{(n)}}, \Omega_{n,i}$ be the measure and the constant from Proposition~\ref{prop:laponept}. There is a constant $D_\Gamma$ such that 
$$\left|\sum_{i=1}^{N(n, \Gamma)}\left(\Theta_{\Gamma^{(n)}_i} + \Omega_{n,i} \delta_{w^{(n)}_i}\right)\right|< D_\Gamma\ .$$
\end{lemma}

\noindent\textbf{Remark:} The constant $D_\Gamma$ in this Lemma depends only on $\Gamma$, and not on $\Gamma^{(n)}$ or its partitions.

\begin{proof}
Let $\Theta_{\Gamma^{(n)}_i}$ and $\Omega_{n,i}$ be as in the statement of Proposition~\ref{prop:laponept} for the respective $\Gamma^{(n)}_i$. To obtain the constant $D_{\Gamma}$, note that the constant in Proposition \ref{prop:laponept}, part 3 depends only on the maximum valence and number of edges in $\Gamma^{(n)}_i$; the valence is certainly no more than $\max_{P\in \Gamma} v_\Gamma(P)$, and similarly $E_{\Gamma^{(n)}_i} \leq E_\Gamma$. Therefore we can take $$D_\Gamma= K(\Gamma) \cdot \left(\sum_{P\in \Gamma} (v(P)-2) + (E_\Gamma +1) \max_{P\in \Gamma} v_\Gamma(P)\right)\ ,$$ where $K(\Gamma)$ is the constant from Lemma \ref{lem:treelemma}.
\end{proof}

By our choice of decomposition of $\Gamma^{(n)}$, we have \begin{align}\label{eq:decomplap}\Delta_{\Gamma^{(n)}} = \Delta_{\Gamma^{(n)}_0} + \sum_{i=1}^{N(n, \Gamma)} \Delta_{\Gamma^{(n)}_i}\ .\end{align} To see this, note that while the various components $\Gamma^{(n)}_i$ may intersect $\Gamma^{(n)}_0$ at a point $P$, the collection $T_P \Gamma^{(n)}$ is accounted for by taking the Laplacians on \emph{all} of the components. We can therefore compute the Laplacian of $f_n$ on $\Gamma^{(n)}$:

\begin{lemma}\label{lem:lapgamman}
We have that $$\Delta_{\Gamma^{(n)}}(f_n) = -2(r_{\pberk, \Gamma^{(n)}_0})_* \nu_{\phi^n} + 2\delta_{r_{\Gamma^{(n)}_0}(\infty)} + \dfrac{2}{d^n-1}\Lambda_n\ ,$$ where $\Lambda_n$ is a measure supported on $\Gamma^{(n)}$ such that $|\Lambda_n| < D_\Gamma$.
\end{lemma}

\begin{proof}
Combine Proposition \ref{prop:lapgamma0} and Lemma \ref{lem:lapgammak}, together with the decomposition of the Laplacian given in (\ref{eq:decomplap}).
\end{proof}

Finally we have

\begin{lemma}
$\Delta_{\Gamma} (f_n) = -2(r_{\pberk, \Gamma})_* \nu_{\phi^n} + 2\delta_{r_{\Gamma}(\infty)} + \dfrac{2}{d^n-1} (r_{\pberk, \Gamma})_* \Lambda_n\ .$
\end{lemma}
\begin{proof}
Note that $\Delta_\Gamma(f_n) = (r_{\Gamma^{(n)}, \Gamma})_* \Delta_{\Gamma^{(n)}}(f_n)$. Note that any path connecting a point $x\in \Gamma_{\widehat{\textrm{FR}}, n}$ to a point $y\in \Gamma$ must intersect $\Gamma^{(n)}_0$; consequently, $$(r_{\Gamma^{(n)}, \Gamma})_*(r_{\Gamma_{\widehat{\textrm{FR}}, n}, \Gamma^{(n)}_0})_* \nu_{\phi^n} = (r_{\Gamma_{\widehat{\textrm{FR}}, n}, \Gamma})_* \nu_{\phi^n}\ .$$ Since the support of $\nu_{\phi^n}$ is contained in $\Gamma_{\widehat{\textrm{FR}}, n}$, this gives 
\begin{equation}\label{eq:doubleretractionnu}(r_{\Gamma^{(n)}, \Gamma})_*(r_{\Gamma_{\widehat{\textrm{FR}}, n}, \Gamma^{(n)}_0})_* \nu_{\phi^n}  = (r_{\pberk, \Gamma})_* \nu_{\phi^n}\ .\end{equation} In a similar manner, we find that \begin{equation}\label{eq:doubleretractionptmass} (r_{\Gamma^{(n)}, \Gamma})_* \delta_{r_{\Gamma^{(n)}_0}(\infty)} = \delta_{r_{\Gamma}(\infty)}\ . \end{equation} Finally, using the fact that the support of $\Lambda_n$ lies in $\Gamma^{(n)}$, we find
\begin{equation}\label{eq:doubleretractionLambda}
(r_{\Gamma^{(n)}, \Gamma})_* \Lambda_n = (r_{\pberk, \Gamma})_* \Lambda_n\ .
\end{equation} Combining (\ref{eq:doubleretractionnu}), (\ref{eq:doubleretractionptmass}), and (\ref{eq:doubleretractionLambda}), and using the decomposition of $\Delta_{\Gamma^{(n)}}(f_n)$ given in Lemma~\ref{lem:lapgamman} yields the asserted result.
\end{proof}

From these results we obtain the proposition that will facilitate the weak convergence. 

\begin{prop}\label{prop:lapgamma}
For $\Gamma$ a fixed finite subtree in $\hberk$ having type II endpoints, 
\begin{align*}
\Delta_{\Gamma}\left(\dfrac{1}{d^{2n}-d^n} \right.&\ord\Res_{\phi^n}(\cdot) - g_{\mu_{\phi}}(\cdot, \cdot)\bigg) \\ 
& = 2(r_{\pberk, \Gamma})_*(\mu_{\phi}-\nu_{\phi^n}) + \dfrac{2}{d^n-1} (r_{\pberk, \Gamma})_*\Lambda_n\ .
\end{align*}
\end{prop}
\begin{proof}
Using the decomposition of $g_{\mu_{\phi}}(x, x) = -\log(\delta(x, x)_{\infty}) + 2\hat{h}_{\phi}(x) + M$ (where $M$ is the constant $\log_v|\Res(\Phi)|^{-1/(d(d-1))}$), we can write $$\dfrac{1}{d^{2n}-d^n} \ord\Res_{\phi^n}(x) - g_{\mu_{\phi}}(x,x) = f_n - 2\hat{h}_{\phi}(x) -M\ .$$ Taking Laplacians on $\Gamma$, we obtain
\begin{align*}
\Delta_{\Gamma}\left(\dfrac{1}{d^{2n}-d^n} \ord\Res_{\phi^n}(\cdot) - g_{\mu_{\phi}}(\cdot, \cdot)\right) & = \Delta_\Gamma (f_n - 2\hat{h}_{\phi} -M)\\
& = \Delta_{\Gamma}(f_n) - 2\Delta_{\Gamma}(\hat{h}_{\phi})\\
& = \left(-2(r_{\pberk, \Gamma})_* \nu_{\phi^n} + 2 \delta_{r_{\Gamma}(\infty)} + \dfrac{2}{d^n-1} (r_{\pberk, \Gamma})_* \Lambda_n\right)\\
& \hspace{1cm}- 2(r_{\pberk, \Gamma})_*(\delta_{\infty} - \mu_{\phi})\\
& = 2(r_{\pberk, \Gamma})_*(\mu_{\phi}-\nu_{\phi^n}) + \dfrac{2}{d^n-1} (r_{\pberk, \Gamma})_*\Lambda_n\ .
\end{align*}

\end{proof}

\subsection{Proof of Convergence}\label{sect:lapweakconv}
We are now ready to prove Theorem \ref{thm:wkconvexplicit}, after which we readily obtain the proof of Theorem \ref{thm:wkconv}.

\begin{proof}[Proof of Theorem \ref{thm:wkconvexplicit}]
Let $\Gamma$ be a finite subtree in $\hberk$, and $f\in$CPA($\Gamma$). Let $R_\Gamma$ be chosen so that $\Gamma \subseteq \mathcal{B}_\rho(\zetaG, R_\Gamma)$. We are interested in estimating $|\int_{\pberk} f d(\mu_\phi - \nu_{\phi^n})|$. Since $f\in \CPA(\Gamma)$, it suffices to estimate $$\left| \int_\Gamma f d (r_{\pberk, \Gamma})_*(\mu_{\phi} - \nu_{\phi^n})\right|\ .$$ From Proposition~\ref{prop:lapgamma}, we can express the measure as $$(r_{\pberk, \Gamma})_* (\mu_{\phi} - \nu_{\phi^n}) =\Delta_{\Gamma} \left( \dfrac{1}{d^{2n}-d^n} \ord\Res_{\phi^n} - g_{\mu_{\phi}}\right) - \dfrac{2}{d^n-1} (r_{\pberk, \Gamma})_*\Lambda_n\ .$$  Thus we can decompose our integral and estimate:

\begin{align*}
\left| \int_\Gamma f d (r_{\pberk, \Gamma})_*(\mu_{\phi} - \nu_{\phi^n})\right| & = \left| \int_\Gamma f d\Delta_{\Gamma}\left(\dfrac{1}{d^{2n}-d^n} \ord\Res_{\phi^n} - g_{\mu_{\phi}}\right) - \dfrac{2}{d^{n}-1}\int_\Gamma f d(r_{\pberk, \Gamma})_* \Lambda_n\right|\\
& \leq \left| \int \left(\dfrac{1}{d^{2n}-d^n} \ord\Res_{\phi^n} -g_{\mu_{\phi}}\right) d\Delta_{\Gamma}(f)\right| + \dfrac{2}{d^n-1} \left| \int_{\Gamma} fd(r_{\pberk, \Gamma})_* \Lambda_n\right|\\
& \leq \max_{\Gamma} \left|\dfrac{1}{d^{2n}-d^n} \ord\Res_{\phi^n} - g_{\mu_{\phi}}\right| \cdot \left|\Delta\right|(f) + \dfrac{2}{d^n-1} \max_{\Gamma}|f| \cdot D_\Gamma \ .
\end{align*} Using the explicit estimate from Theorem \ref{thm:fnconvexplicit}, we find 
\begin{align*}
\left| \int_\Gamma f d (r_{\pberk, \Gamma})_*(\mu_{\phi} - \nu_{\phi^n})\right| & \leq \dfrac{4}{d^n-1}\left(\max \left( C_\phi, R_\Gamma\right)\cdot \left| \Delta\right|(f) + \max_\Gamma \left|f\right| \cdot D_\Gamma\right)\ .
\end{align*}

\end{proof}

With this we are able to prove Theorem \ref{thm:wkconv}, that the measures $\left\{\nu_{\phi^n}\right\}$ converge weakly to $\mu_\phi$. In order to show weak convergence, we will need that for all choices of $F\in \mathcal{C}(\pberk)$, we have $$\int_{\pberk} Fd(\mu_{\phi} - \nu_{\phi^n}) \to 0$$ as $n\to \infty$. Here is the proof.

\begin{proof}[Proof of Theorem \ref{thm:wkconv}]

Let $\epsilon > 0$. Choose $F\in \mathcal{C}(\pberk)$. By Proposition 5.4 in \cite{BR}, we know that there exists a finite subtree $\Gamma$ and a function $f\in$ CPA($\Gamma$) such that $$\sup_{\pberk} |F(x)-f\circ r_{\pberk, \Gamma}(x)| < \dfrac{\epsilon}{4}\ .$$ Since both $\mu_{\phi}$ and $\nu_{\phi^n}$ are both probability measures, we have that 

\begin{align*}
\left|\int_{\pberk}Fd(\mu_{\phi}-\nu_{\phi^n})\right| & = \left| \int_{\pberk} \left( F-f\circ r_{\pberk, \Gamma}\right)d(\mu_{\phi}-\nu_{\phi^n}) + \int_{\pberk} f\circ r_{\pberk, \Gamma} d(\mu_{\phi}-\nu_{\phi^n})\right|\\
& \leq \dfrac{\epsilon}{2} + \left|\int_{\pberk} f\circ r_{\pberk, \Gamma} d(\mu{\phi} - \nu_{\phi^n})\right|\\
& = \dfrac{\epsilon}{2} + \left|\int_{\Gamma} f d (r_{\pberk, \Gamma})_* (\mu_{\phi} - \nu_{\phi^n})\right|\ .
\end{align*}

Since $f$ and $\Gamma$ are fixed,  Theorem \ref{thm:wkconvexplicit} tells us that for $n$ sufficiently large, the remaining integral term can be made smaller than $\dfrac{\epsilon}{2}$. This establishes weak convergence.
\end{proof}

\subsection{Examples}\label{sect:exampleswkconv}
We now give several explicit examples showing the weak convergence of the crucial measures.

\begin{ex}\label{ex:wkconvHaar} Let $K=\mathbb{C}_p$ for some prime $p\geq 3$, and let $\phi(z) = \dfrac{z^p-z}{p}$. It is known (see \cite{BR}, Example 10.120) that the invariant measure attached to $\phi$ is the Haar measure on $\mathbb{Z}_p:= \lim_{\leftarrow} \mathcal{O}_K / (\mathfrak{m}_K^n)$. The classical fixed points of $\phi$ are $\infty$ and points $\zeta_1, ..., \zeta_p$, where the $\zeta_i$ lie in the different cosets of $\mathbb{Z}_p / p\mathbb{Z}_p$; we have that $\Gamma_{\textrm{Fix}}$ is the tree spanned by the $\zeta_i$ and $\infty$. The Gauss point $\zetaG$ is a non-fixed branch point of $\Gamma_{\textrm{Fix}}$, with valence $p+1$; hence $w_\phi(\zetaG) = p+1-2 = p-1 = \deg(\phi) -1$. By the weight formula, it is the only weighted point.

We now look to preimages of $\zetaG$ under $\phi$. The set $\phi^{-1}(\zetaG)$ is a collection of disjoint discs $D(a_1, r_1), ..., D(a_p, r_p)$ where $a_i$ lie in the various directions towards fixed points, and the $r_i$ can all be taken to be $1/p$. To see this, note that the preimages of zero are the points satisfying $a_i^p - a_i = 0$; in the reduction modulo $\mathfrak{m}_K$ these are the same as the classical fixed points $\zeta_i$ above. For the radii, one checks that $|\phi(a_i+p) - \phi(a_i)| = 1$, which establishes that $r_i = |p|$ for each $i$. It follows from these two facts that the discs are disjoint.

More generally, we claim that a point in the $n^{th}$ preimage of $\zetaG$ corresponds to a disc $$D(a_{i_n i_{n-1} \cdots i_2 i_1}, p^{-n})$$ where the points \{$a_{i_n i_{n-1} \cdots i_2 i_1}\}$ are the successive preimages of $0$ indexed in such a way that $\phi(a_{i_n i_{n-1} \cdots i_2 i_1}) = a_{i_{n-1} \cdots i_2 i_1}$, $\phi^{(2)}(a_{i_n i_{n-1} \cdots i_2 i_1}) = a_{i_{n-2} \cdots i_2 i_1}$, ... and finally $\phi^{(n)}(a_{i_n i_{n-1} \cdots i_2 i_1}) = 0$. Note that we make no assertion as to whether the points $a_{i_n \cdots i_1}$ are distinct, though in the end we will deduce that in fact they are:

\noindent \emph{Claim:} The points $a_{i_n \cdots i_1}$ lie in distinct coset classes modulo $p^n$. 

Note that we have already seen this for $n=1$, where the preimages of $0$ are $a_1 = \zeta_1, ..., a_p = \zeta_p$. We proceed now by induction. Suppose $a_{i_n\cdots i_1} \equiv a_{j_n \cdots j_1}\mod p^n$. By the relation $\phi(a_{i_n \cdots i_1}) = a_{i_{n-1} \cdots i_1}$, we find that $$a_{i_n\cdots i_1} = a_{i_n \cdots i_1}^p - p\cdot a_{i_{n-1} \cdots i_1}\ ,$$ and likewise $$a_{j_n \cdots j_1} = a_{j_n \cdots j_1}^p - p\cdot a_{j_{n-1} \cdots j_1}\ .$$ The congruence $a_{i_n \cdots i_1} \equiv a_{j_n \cdots j_1} \mod p^n$ implies that for some $r\in \mathbb{Z}$ we have $$p^n r = a_{i_n \cdots i_1}^p - a_{j_n \cdots j_1}^p + p (a_{j_{n-1} \cdots j_1} - a_{i_{n-1} \cdots i_1})\ .$$ By our induction hypothesis, $p^{n-1} \nmid a_{j_{n-1} \cdots j_1} - a_{i_{n-1} \cdots i_1}$, and therefore we conclude that $p^{n} \nmid a_{i_n \cdots i_1}^p - a_{j_n \cdots j_1}^p$. But this contradicts that $a_{i_n \cdots i_1} \equiv a_{j_n \cdots j_1} \mod p^n$, and thus establishes the claim.

In particular, the above claim implies that the preimages $\phi^{-n}(0)$ are in one-to-one correspondence with the coests of $\mathbb{Z}_p / p^n \mathbb{Z}_p$. The fact that the radii corresponding to the $n^{th}$ preimages of $\zetaG$ are $p^{-n}$ can be seen by the fact that $a_{i_n \cdots i_1} + p^n$ maps to a point lying at distance $p^{n-1}$ from $a_{i_{n-1} \cdots i_1}$. 

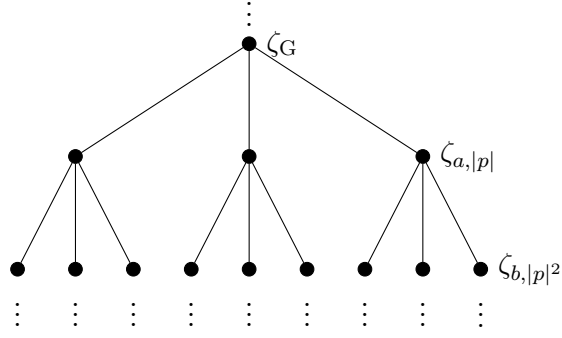
\begin{figure}\label{fig:Haar}

\begin{tikzpicture}[level 1/.style = {sibling distance = 6em}, level 2/.style = {sibling distance = 2em}]
\node[shape=circle,draw, fill=black!100, label = right:$\zetaG$, scale={0.5}, label=above:$\vdots$]{}
   child {node [shape=circle, draw, fill=black!100, scale={0.5}]{}
      child {node [shape=circle,draw, fill=black!100, scale={0.5}, label=below:$\vdots$]{}}
      child {node [shape=circle,draw, fill=black!100, scale={0.5}, label=below:$\vdots$]{}}
      child {node [shape=circle,draw, fill=black!100, scale={0.5}, label=below:$\vdots$]{}}
}
   child {node [shape=circle,draw, fill=black!100, scale={0.5}]{}
      child {node [shape=circle,draw, fill=black!100, scale={0.5}, label=below:$\vdots$]{}}
      child {node [shape=circle,draw, fill=black!100, scale={0.5}, label=below:$\vdots$]{}}
      child {node [shape=circle,draw, fill=black!100, scale={0.5}, label=below:$\vdots$]{}}
}
   child {node [shape=circle,draw, fill=black!100, scale={0.5}, label=right:$\zeta_{a,|p|}$]{}
      child {node [shape=circle,draw, fill=black!100, scale={0.5}, label=below:$\vdots$]{}}
      child {node [shape=circle,draw, fill=black!100, scale={0.5}, label=below:$\vdots$]{}}
      child {node [shape=circle,draw, fill=black!100, label = right:$\zeta_{b, |p|^2}$, scale={0.5}, label=below:$\vdots$]{}}
};

\end{tikzpicture}
\caption[An Example of Weak Convergence: $\frac{z^p-z}{p}$]{An example when $p=3$ and $n=3$. The bold vertices are points in the support of $\nu_{\phi^3}$. Each has valence 4 in the tree  $\Gamma_{\textrm{Fix},3}$ spanned by the classical fixed points, and each of these points is moved by $\phi$; hence these points each have weight $\frac{1}{13}$.}
\end{figure}

We next claim that each point $\zeta\in \bigcup_{i=0}^{n-1} \phi^{-i}(\zetaG)$ is a branch point of $\Gamma_{\textrm{FR},n}$ with valence at least $p+1$. To see this, note that $D(a_{i_k \cdots i_1}, p^{-k})\subseteq D(0,1)$, and that $\phi^k$ maps $D(a_{i_k \cdots i_1}, p^{-k})$ onto $D(0,1)$. Then since $\phi^{k+1}(D(a_{i_{k+1} i_k \cdots i_1}, p^{-(k+1)})) = D(0,1)$, it must contain a fixed point of $\phi^{k+1}$, which necessarily lies in $D(a_{i_k \cdots i_1}, p^{-k})$. There are $p$ such discs $D(a_{i_{k+1} \cdots i_1}, p^{-(k+1)})$ lying in $D(a_{i_k\cdots i_1}, p^{-k})$ (corresponding to the cosets of $\mathbb{Z}_p / p^{k+1} \mathbb{Z}_p$ in $D(a_{i_k\cdots i_1}, p^{-k})$). This, together with the direction towards $\infty$ implies that $\zeta_{a_{i_k \cdots i_1}, p^{-k}}$ has valence at least $p+1$ in $\Gamma_{\textrm{FR}, n}$ for each $k=1, ..., n-1$. 

Finally, note that each point $\zeta\in \bigcup_{i=0}^{n-1} \phi^{-i}(\zetaG)$ is moved, hence $w_{\phi^{(n)}}(\zeta)= v_{\Gamma_{\textrm{FR},n}}(\zeta)-2 \geq p-1$. There are $\sum_{i=0}^{n-1} p^i = \frac{p^n-1}{p-1}$ such points, and by summing the corresponding weights, we find a total weight of at least $p^n-1$. As this is equal to $\deg(\phi^n) -1$, these are the only points which bear weight and their weights must be eactly equal to $p-1$. 

It also follows from the above remarks that the points $\phi^{-n}(\zetaG)$ distribute themselves equally among the representatives of $\mathcal{O}_K / (\mathfrak{m}_K^n)$, and so as $n$ tends to infinity these points converge to the points of $\mathbb{Z}_p$. For each fixed $k$, we know $\mu_{\textrm{Haar}}(D(a, p^{-k})) = p^{-k}$ for any center $a\in \ZZ_p$. The corresponding $\nu_{\phi^n}$ measure can be computed by considering the convex hull of $D(a, p^{-k})$ in $\pberk$, which we will denote by $\mathcal{D}(a, p^{-k})$ (concretely, this is $\pberk \setminus B_{\zeta_{a, p^{-k}}}(\vv_\infty)^-$). The set $\mathcal{D}(a, p^{-k})$ will only receive $\nu_{\phi^n}$ weight when $n\geq k$, and in this case, each point in $\mathcal{D}(a, p^{-k})$ that receives $\nu_{\phi^n}$-mass will have weight $\frac{p-1}{p^n-1}$. We need only count how many such points there are in a given $\mathcal{D}(a, p^{-k})$. 

It will suffice to assume our center is of the form $a_{i_k i_{k-1} \cdots i_1}$, since the discs centered at these points of radius $p^{-k}$ form a partition of $\ZZ_p$. From the arguments above, $\mathcal{D}(a_{i_k i_{k-1} \cdots i_1}, p^{-k})$ contains $1+p+p^2 + ... + p^{(n-1)-k} = \frac{p^{n-k}-1}{p-1}$ points which receive $\nu_{\phi^n}$-mass, corresponding to the preimages $\phi^{-k}(\zetaG), \phi^{-(k+1)} (\zetaG), \dots, \phi^{-(n-1)}(\zetaG)$ which lie in $\mathcal{D}(a_{i_k i_{k-1} \cdots i_1})$ (there are $1, p, p^2, ... p^{(n-1)-k}$ such points, resp.). Therefore, $$\nu_{\phi^n}(\mathcal{D}(a_{i_k i_{k-1} \cdots i_1}, p^{-k})) = \frac{p-1}{p^n-1}\cdot  \frac{p^{n-k}-1}{p-1} \to \frac{1}{p^k} = \mu_{\textrm{Haar}}(\mathcal{D}(a_{i_k i_{k-1} \cdots i_1}, p^{-k})) \ ,$$ which completes the proof.

\end{ex}

\begin{ex}[Latt\`es Maps]Let $K=\CC_p$ be the $p$-adic complex numbers, and fix $q\in K$ satisfying $0< |q| < 1$. A Tate curve $E/K$ is an elliptic curve which is isomorphic (as a group) to the quotient $K^\times / q^\ZZ$ (see, e.g., \cite{AEC} Appendix C). In particular, the multiplication-by-$m$ map $[m]:E\to E$ is given by the quotient of the map $z\mapsto z^m$ on $K^\times$. Viewing $E$ as an affine plane curve $y^2=x^3 + Ax+B$ for $A,B\in K$, we let $\pi:E\to \mathbb{P}^1(K)$ be the map $(x,y)\mapsto x$. The \emph{Latt\`es map} corresponding to multiplication by $m$ is the rational map $\phi_m$ of degree $m^2$ on $\PP^1(K)$ which completes the diagram

\begin{center}
$\begin{CD}
E @>[m]>> E\\
@VV\pi V   @VV\pi V\\
\PP^1(K) @>\phi_m>> \PP^1(K)
\end{CD}$
\end{center}

The map $\phi^m$ can be extended to the Berkovich line $\pberk$ over $K$ in a natural way. Choosing suitable coordinates on $\PP^1(K)$, the Julia set of $\phi_m$ is the segment $\mathcal{J} = [\zetaG, \zeta_{0, |q|^{-1/2}}]$, and the equilibrium measure is the uniform measure on this segment (\cite{FRL}, Proposition 5.1). For $i=0, 1, ..., m-1$, let $\mathcal{I}_i = [\zeta_{0, |q|^{-i/2m}}, \zeta_{0, |q|^{-(i+1)/2m}}]$, so that $\mathcal{J} = \bigcup_{i=0}^{m-1} \mathcal{I}_i$. The restriction of $\phi_m$ to $\mathcal{I}_i$ is an affine map (with respect to the metric $\rho$), sending $\mathcal{I}_i$ bijectively onto $\mathcal{J}$. Along each segment $\mathcal{I}_i$, the map $\phi_m$ has slope $(-1)^i m$ (\cite{FRL}, Proposition 5.1).
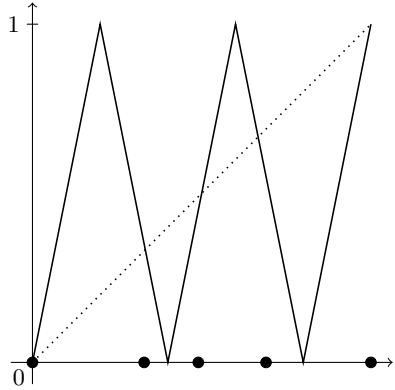
\begin{figure}
\begin{tikzpicture}
\datavisualization [school book axes, x axis = {length = 4.5cm}, y axis = {length = 4.5cm}, visualize as line = diag, visualize as line = sawtooth, visualize as scatter = fixedpts, diag = {style ={dotted}}, fixedpts = {style = {mark = *}}]
data [set=sawtooth] {
	x, y
	0, 0
	0.1, 1
	0.2, 0
	0.3, 1
	0.4, 0
	0.5, 1
}
data [set=diag] {
	x, y
	0,0
	0.5,1
}
data [set=fixedpts]{
	x,y
	0,0
	0.165, 0
	0.245,0
	0.345,0
	0.5,0
};
\end{tikzpicture}
\caption[An Example of Weak Convergence: Latt\`es Maps] {The action of $\phi_5$ on $\mathcal{J}$ is represented by the sawtooth graph to the left. The bold points along the $x$-axis are the type II fixed points. Those points corresponding to edges where the graph is decreasing have 2 shearing directions, while those corresponding to edges where the graph is increasing have no shearing.}
\end{figure}

Each interval $\mathcal{I}_i$ contains a unique type II fixed point of $\phi_m$ which we denote by $\zeta_i$, for $i=0, 1, ..., m-1$. Since $\phi_m$ has slope $(-1)^i m$, the rate of repulsion $r_{\phi_m}(\zeta_i, \vv) = m$ for any direction $\vv\in T_{\zeta_i}$ pointing into $\mathcal{J}$. Hence the points $\zeta_i$ are repelling fixed points with $\deg_{\phi_m}(\zeta_i) \geq m$. 

The type I fixed points of $\phi_m$ all lie in branches off of the points $\zeta_i$. For $i=1, 2, ..., m-2$, the points $\zeta_i$ each have two direction $\vv_0^{(i)}, \vv_\infty^{(i)}\in T_{\zeta_i}$ pointing into $\mathcal{J}$ which contain classical fixed points. If $i$ is odd (so that the slope of $\phi_m$ on $\mathcal{I}_i$ is $-m<0$), these directions are flipped by $(\phi_m)_*$; hence $\zeta_i$ has at least two shearing directions. If $i$ is even, $(\phi_m)_*$ fixes these directions at the respective $\zeta_i$. 

At the endpoints, $\zeta_0$ has a unique direction $\vv_\infty^{(0)}\in T_{\zeta_0}$ pointing into $\mathcal{J}$ which contains a type I fixed point, and it is fixed by $(\phi_m)_*$. The point $\zeta_{m-1}$ also has only one direction $\vv_{0}^{(m-1)}$ containing a type I fixed point; if $m$ is even, this direction is mapped to $\vv_\infty\in T_{\zeta_{m-1}}$ by $(\phi_m)_*$ (since the slope of $\phi_m$ on $\mathcal{I}_{m-1}$ is negative when $m$ is even). If $m$ is odd, then the direction $\vv_0^{(m-1)}$ is fixed by $(\phi_m)_*$. Therefore, for $m$ even we have 
\begin{equation}\label{eq:lowerbdweighteven}w_{\phi_m}(\zeta_i) = \deg_{\phi_m}(\zeta_i) -1 +N_{\textrm{Shearing}}(\zeta_i) \geq\left\{
\begin{matrix} 
m+1,& i=1,3,5,...m-3\\
m, & i=m-1\\
m-1, & i=0,2,4,...,m-2
\end{matrix}\right. \ .
\end{equation} Summing the lower bounds over each of the points $\zeta_i$ gives 
\begin{align*}
\sum_{i=0}^{m-1} w_{\phi_m}(\zeta_i) &\geq (m+1)\left(\frac{m-2}{2}\right) + m + (m-1)\left(\frac{m}{2}\right)\\
& = m^2-1\ .
\end{align*} Since the sum of the weights is always equal to $\deg(\phi_m)-1 = m^2-1$, the lower bounds given in (\ref{eq:lowerbdweighteven}) must be equalities.

Similarly, if $m$ is odd we have
\begin{equation}\label{eq:lowerbdweightodd}w_{\phi_m}(\zeta_i) = \deg_{\phi_m}(\zeta_i) -1 +N_{\textrm{Shearing}}(\zeta_i) \geq\left\{
\begin{matrix} 
m+1,& i=1,3,5,...m-2\\
m-1, & i=0,2,4,...,m-1
\end{matrix}\right. \ .
\end{equation} Again summing the lower bounds over each of the $\zeta_i$ gives 
\begin{align*}
\sum_{i=0}^{m-1} w_{\phi_m}(\zeta_i) &\geq (m+1)\left(\frac{m-1}{2}\right) + (m-1)\left(\frac{m+1}{2}\right)\\
&  = m^2-1\ .
\end{align*} Since the sum of the weights is $m^2-1$, we again conclude that the lower bounds in (\ref{eq:lowerbdweightodd}) are equalities. 

Since the iterates of $\phi_m$ satisfy $\phi_m^{(n)} = \phi_{m^n}$, we see that for fixed $m$, the points $\zeta_i$ distribute themselves uniformly (with respect to the weight functions) along the interval $[\zetaG, \zeta_{0, |q|^{-1/2}}]$ as $n\to\infty$.

\end{ex}

\section{Barycenters}\label{sect:barycenters}
\subsection{Notation and Summary of Results}
In this section we show that the sets $\MinResLoc(\phi^n)$ `approach' the barycenter of the canonical measure $\mu_\phi$. We first recall the definition of the barycenter of a measure in $\pberk$:

\begin{defn}\label{defn:barycenter} (Rivera-Letelier)
Let $\nu$ be a finite, positive Radon measure on $\pberk$. The barycenter of $\nu$, denoted $\Bary(\nu)$, is the collection of points $Q\in \pberk$ such that $\nu(\B_Q(\vec{v})^-) \leq \frac{1}{2} \nu(\pberk)$ for each $\vec{v}\in T_Q$.
\end{defn}

\begin{ex} 
\begin{enumerate}
\item Let $p$ be an odd prime and let $K=\mathbb{C}_p$, and let $\mu$ denote the Haar measure on $\mathbb{Z}_p$. Here the barycenter is $\{\zetaG\}$. To see this, note that each coset $k+p\mathbb{Z}_p$ has $\mu(k+p\mathbb{Z}_p) = \frac{1}{p}$. Letting $\vv_1, ..., \vv_p$ denote the directions at $\zetaG$ corresponding to these cosets, we have $\mu(B_{\zetaG}(\vv_i)^-) = \frac{1}{p}$. In particular, if $Q\neq \zetaG$ and $\vv_G\in T_Q$ is the direction towards $\zetaG$, then $\mu(B_Q(\vv_G)^-) \geq (p-1)\frac{1}{p} >\frac{1}{2}$, where the final inequality holds because $p$ is odd. Thus $Q\not\in \Bary(\mu)$.

\item Let $\nu= \frac{1}{2}\delta_A + \frac{1}{2} \delta_B$. Then the barycenter of $\nu$ is precisely the segment $[A,B]$. Let $\vv_B\in T_A$ be the direction pointing towards $B$. This example also shows that $\nu(\cup_{\vv\in T_A\setminus\{\vv_B\}} B_A(\vv)^-)$ need not equal $\frac{1}{2}$.

\item Let $K=\mathbb{C}_p$. The barycenter of the canonical measure attached to $\phi(T) = \frac{T^2-1}{p}$ is the interval $[\zeta_{D(1, \frac{1}{p})}, \zeta_{D(-1, \frac{1}{p})}]$; a proof will be given in Example~\ref{ex:Benedettos} below. This example is due to Rob Benedetto (personal communication).

\item Let $\nu = \delta_{\zeta}$ be a point mass at some point $\zeta\in \hberk$. Then the barycenter for $\nu$ is $\zeta$ itself.

\end{enumerate}
\end{ex}

The main result of this section is 
\begin{prop}\label{prop:minresepsilon}
Let $\phi\in K(z)$ be a rational map of degree $d\geq 2$. For any $\epsilon >0$, there exists an $N$ such that for every $n\geq N$ we have $$\MinResLoc(\phi^n) \subseteq B_\rho(\Bary(\mu_\phi), \epsilon)\ .$$
\end{prop}

\subsection{Preliminary Results on Arakelov-Green's Functions}
\smallskip
We will restrict our attention to probability measures $\nu$ on $\pberk$. In this section, we will establish several important facts about the barycenter of a probability measure on $\pberk$, all of which are due to Rivera-Letelier but which have yet to be published. As they are essential to the proof of Proposition~\ref{prop:minresepsilon}, we include our own proofs here. More explicitly, we will show (i) the barycenter of a probability measure is always non-empty (Proposition~\ref{prop:RL}), (ii) it is always a point or a segment (Proposition~\ref{prop:RL}), and (iii) the associated diagonal Arakelov-Green's function $g_\nu(x,x)$ attains its minimum precisely on $\Bary(\nu)$ (Proposition~\ref{prop:RL}). Having established these preliminary properties of the Arakelov-Green's functions, we use them to establish results relating the sets $\MinResLoc(\phi^n)$ to $\Bary(\mu_\phi)$. 

Throughout this section, we are primarily interested in the diagonal values of the Arakelov-Green's function, and we will write $g_\nu(x) := g_\nu(x,x)$. Recall that we are also assuming that probability measures are \emph{Radon} measures.

\begin{lemma}\label{lem:slopegreensfcn}
Let $\nu$ be a probability measure on $\pberk$. 

\begin{enumerate}
\item[(A)] The function $g_\nu(\cdot)$ is convex up along segments in $\pberk$.
\item[(B)] If $Q\in \hberk$, and $\vv$ is  any direction in $T_Q$, we have $$\del_v(g_\nu)(Q) = 1-2\nu(B_Q(\vv)^-)\ .$$
\item[(C)] If $Q$ is a type I point with $c=\nu(\{Q\})$, then for any $\epsilon >0$ there exists a type II point $Q_0$ sufficiently close to $Q$ such that if $\vv\in T_{Q_0}$ denotes the direction towards $Q$, we have $$|(1-2c) -  \del_{\vv}(g_{\nu})(Q_0) |<\epsilon\ .$$ In particular, if $\nu$ does not charge $Q$, then $$|1-\del_{\vv}(g_\nu)(Q_0)|< \epsilon \ .$$
\end{enumerate}
\end{lemma}

\begin{proof}
Fix $Q\in \hberk$. Recall that the Arakelov-Green's function is given by 
\[
g_\nu(x) = \int -\log_v \delta(x,x)_\zeta d\nu(\zeta) + C
\] for an appropriate normalizing constant $C$. The Hsia kernel satisfies a change of variables formula (see \cite{BR} Equation (4.29))
\[
\delta(x,y)_\zeta = \frac{1}{\delta(\zeta, Q)_{\zetaG}^2}\cdot \frac{\delta(x,y)_Q}{\delta(x,\zeta)_Q \cdot \delta(y, \zeta)_Q}\ .
\] Inserting this into the expression for $g_\nu(x)$ we obtain
\begin{equation}\label{eq:partialgreendecomp}
g_\nu(x) = -\log_v \delta(x,x)_Q + 2\int \log_v\delta(x,\zeta)_Q d\nu(\zeta) + 2\int \log_v\delta(\zeta, Q)_{\zetaG}d\nu(\zeta) + C\ .
\end{equation}We next re-write some of the Hsia kernel appearing here in terms of the path distance metric; to do so, we recall that (see \cite{BR} Equation (4.31)) $-\log_v \delta(x,y)_\xi = \rho(x\wedge_\xi y, \xi) + \log_v \delta(\xi, \xi)_{\zetaG}$. Then (\ref{eq:partialgreendecomp}) becomes
\[
g_\nu(x) = \rho(x, Q) - 2\int \rho(x \wedge_Q \zeta, Q)d\nu(\zeta) - \log_v \delta(Q, Q)_{\zetaG} + 2\int \log_v \delta(\zeta, Q)_{\zetaG} d\nu(\zeta) + C\ .
\] Finally, by reversing the change-of-variables formula for the Hsia kernel, we see that the last three terms in the previous expression are simply $g_\nu(Q)$. Therefore, we obtain
\begin{equation}\label{eq:rewritteng}
g_\nu(x) = \rho(x,Q) - 2\int \rho(x\wedge_Q \zeta, Q) d\nu(\zeta) + g_\nu(Q)\ .
\end{equation}

To prove (A), fix a segment $[P,Q]\subseteq \pberk$, and fix also a point $R\in [P,Q]$. We will show that
\[
\frac{g_\nu(P) - g_\nu(Q)}{\rho(P,Q)} \geq \frac{g_\nu(R) - g_\nu(Q)}{\rho(R,Q)}\ .
\] Applying (\ref{eq:rewritteng}), this is equivalent to
\[
1 - \frac{2}{\rho(P,Q)} \int \rho(P\wedge_Q \zeta, Q) d\nu(\zeta) \geq 1 - \frac{2}{\rho(R,Q)} \int \rho(R\wedge_Q \zeta, Q)d\nu(\zeta)\ .
\] We will prove this by showing that for every $\zeta\in \pberk$, 
\begin{equation}\label{eq:needtoshow}
\frac{\rho(P\wedge_Q\zeta, Q)}{\rho(P,Q)} \leq \frac{\rho(R\wedge_Q \zeta, Q)}{\rho(R,Q)}\ .
\end{equation} As $R\in [P,Q]$, there are two cases. Suppose first that $R\wedge_Q \zeta = P\wedge_Q \zeta$. Then (\ref{eq:needtoshow}) holds by the fact that $\rho(R,Q) \leq \rho(P,Q)$. If $R\wedge_Q \zeta \neq P\wedge_Q \zeta$, then $R\wedge_Q \zeta = R$, so that the right side of (\ref{eq:needtoshow}) is 1. The inequality then holds by the fact that $\rho(P\wedge_Q \zeta ,Q) \leq \rho(P, Q)$. This completes the proof of (A).

We next prove (B). Fix $\vv\in T_Q$. We can evaluate the integral in (\ref{eq:rewritteng}) for a point\footnote{Recall that $Q+t\vv$ is actually an arbitrary point in $B_{\vv}(Q)^-$ satisfying $\rho(Q+t\vv, Q) = t$, and the limit of the difference quotient is unaffected by which point we choose. See the discussion in Section~\ref{sect:convandnot}.} $Q+t\vv$ by retracting to the segment $[Q, Q+t\vv]$. Let $\nu_{[Q, Q+t\vv]} = (r_{\pberk, [Q, Q+t\vv]})_*\nu$ denote the retraction of the measure\footnote{More generally, if $\mathcal{I}\subseteq \hberk$ is an interval, we will let $\nu_\mathcal{I} = (r_{\pberk, \mathcal{I}})_* \nu$ denote the retraction of $\nu$ to $\mathcal{I}$.}; then (\ref{eq:rewritteng}) becomes \begin{align}\label{eq:startdirectionalderivative}g_\nu(Q+t\vv) = \rho(Q, Q+t\vv) - 2\int_Q^{Q+t\vv} \rho(Q, s) d\nu_{[Q, Q+t\vv]}(s) + g_\nu(Q)\ .\end{align}

We will explicitly estimate the quantity $$\dfrac{g_\nu(Q+t_0\vv) - g_\nu(Q) }{t_0} $$ for small values of $t_0$. Let $c_1 = 1-\nu(B_Q(\vv)^-)$; then  the retraction measure $\nu_{[Q, Q+t_0\vv]} $ decomposes as $\nu_{[Q, Q+t_0\vv]} = c_1 \delta_Q + \nu_{(Q, Q+t_0\vv)} + c_2(t_0) \delta_{Q+t_0\vv}$, where $c_2(t_0) \to \nu(B_Q(\vv)^-)$ as $t_0 \to 0$; the existence\footnote{Explicitly, $c_2(t_0)$ can be given as $c_2(t_0):=1-\nu(B_{Q+t_0\vv}(\vv_Q)^-)$, the complementary mass of the ball at $Q+t_0\vv$ pointing towards $Q$.} of $c_2(t_0)$ follows from the regularity of $\nu$, which we assume to be Radon. Inserting this decomposition into Equation (\ref{eq:startdirectionalderivative}), we find
\begin{align} 
g_\nu(Q+t_0\vv) - g_\nu(Q) & = \rho(Q, Q+t_0\vv) - 2\int_Q^{Q+t_0\vv} \rho(Q,t) d\nu_{[Q, Q+t_0\vv]}(t)\nonumber\\
& = t_0 - 2\int_Q^{Q+t_0\vv} \rho(Q,t) d\nu_{(Q, Q+t_0\vv)} -2c_2(t_0)\cdot t_0\label{eq:greensecantslope}\ .
\end{align} Observe that $\nu_{(Q, Q+t_0\vv)}(Q, Q+t_0 \vv)\leq 1-c_1 - c_2(t_0)$, hence we can estimate the  integral in (\ref{eq:greensecantslope}) as $$0 \leq \int_Q^{Q+t_0\vv} \rho(Q, t) d\nu_{(Q, Q+t_0\vv)} \leq (1-c_1-c_2(t_0))\cdot t_0$$ And so for $t_0$ sufficiently small, $$(2c_1-1) \leq \frac{g_\nu(Q+t_0\vv) - g_\nu(Q)}{t_0} \leq 1-2c_2(t_0)\ .$$ Rewriting the left side with the explicit value of $c_1$ we have $$1-2\nu(B_Q(\vv)^-) \leq \frac{g_\nu(Q+t_0 \vv) - g_\nu(Q)}{t_0} \leq 1-2c_2(t_0)\ .$$ Letting $t_0 \to 0$, we have $$\del_v(g_\nu)(Q) = \lim_{t_0\to 0} \frac{g_\nu(Q+t_0 \vv) - g_\nu(Q)}{t_0}  = 1-2\nu(B_Q(\vv)^-)$$ as asserted. \\

We now show part (C). Let $Q$ be a type I point, set $c=\nu(\{Q\})$, and fix $\epsilon >0$. By the regularity of $\nu$, we can find a type II point $Q_0$ so that, if $\vv\in T_{Q_0}$ is the direction towards $Q$, then $c-\frac{\epsilon}{2} \leq \nu(B_{Q_0}(\vv)^-) < c+ \frac{\epsilon}{2}$. Thus $$|1-2c - \del_{\vv}(g_{\nu})(Q_0)| = 2|\nu(B_{Q_0}(\vv)^-) - c| < \epsilon\ .$$ If $Q\not\in \supp(\nu)$ this reduces to $$|1-\del_{\vv}(g_\nu)(Q_0)| < \epsilon$$ as asserted.

\end{proof}

With the above lemma, we can prove the following result about the geometry of the barycenter of a probability measure. This result is due originally to Rivera-Letelier, but the proof given here is our own:

\begin{prop}\label{prop:RL} (Rivera-Letelier)
Let $\nu$ be a probability measure on $\pberk$ with continuous potentials.
\begin{enumerate}
\item The function $g_\nu(x)$ is minimized precisely on $\Bary(\nu)$.
\item $\Bary(\nu)$ is a nonempty subset of $\hberk$, and it either consists of a single point or is a closed segment. 
\end{enumerate}
\end{prop}
\begin{proof}

Since $\nu$ has continuous potentials, the function $g_\nu(x)$ is lower semi-continuous on $\pberk$. Also, $\nu$ does not charge any type I points: if $\nu(\{Q\})>0$ for some type I point $Q$, then we can decompose the potential function as $$u_\nu(z, \zetaG) = -\nu(\{Q\})\log_v\delta(z, Q)_{\zetaG} + \int_{\pberk\setminus\{Q\}} -\log_v\delta(z, w)_{\zetaG} d\nu(w)\ .$$ But then $\lim_{z\to Q} u_\nu(z, \zetaG) = -\infty$, contradicting that $u_\nu(z, \zetaG)$ is continuous as a function to $\mathbb{R}$.

Since $\pberk$ is compact in the weak topology and $g_\nu(x)$ is lower semicontinuous, it must assume a minimum. Moreover, the points at which $g_\nu(x)$ attains its minimum lie in $\hberk$: if it contained a type I point $Q$, then necessarily $\nu(\{Q\}) = 0$ and by Lemma~\ref{lem:slopegreensfcn} there exists a type II point $Q_0$ sufficiently near $Q$ such that if $\vv\in T_{Q_0}$ is the direction towards $Q$, then $\del_{\vv}(g_\nu)(Q_0) > \frac{1}{2}$. In particular, $g_\nu(Q_0) < g_\nu(Q)$, contradicting that $Q$ is a minimum value of $g_\nu(x)$.

If $Q$ is a point at which $g_\nu(x)$ is minimized, then $\del_{\vv}(g_\nu)(Q) \geq 0$ for every $\vv\in T_Q$. In particular, it follows from Lemma~\ref{lem:slopegreensfcn} that $\nu(B_Q(\vv)^-) \leq \frac{1}{2}$ for every $\vv\in T_Q$; thus $Q$ is in the barycenter of $\nu$. It follows that $\Bary(\nu)$ is nonempty. Conversely, suppose that $Q\in \Bary(\nu)$; then $\nu(B_Q(\vv)^-)\leq \frac{1}{2}$ for every $\vv\in T_Q$, and by Lemma~\ref{lem:slopegreensfcn}, we find that $\del_{\vv}(g_\nu)(Q) \geq 0$ for every $\vv\in T_Q$. Now choose any point $P\in \hberk$ at which $g_\nu(x)$ is minimized, and consider the segment $[P, Q]$. Since $\del_{\vv_P}(g_\nu)(Q) \geq 0$, it follows that $g_\nu(Q) \leq g_\nu(P)$; since $g_\nu$ is minimized at $P$, we find that $g_\nu(Q) = g_\nu(P)$, i.e. $g_\nu$ is minimized at $Q$ as well. This finishes the proof that $g_\nu$ is minimized precisely on $\Bary(\nu)$. \\

We now show that $\Bary(\nu)$ is connected. Suppose there are two points $P, Q$ in the barcyenter of $\nu$, and let $R\in [P,Q]$. Then $R$ is also in the barycenter of $\nu$, since for any $\vv\in T_R$ we have either $B_R(\vv)^- \subseteq \B_P(\vv_R)^-$ or $B_R(\vv)^- \subseteq B_Q(\vv_R)^-$, where $\vv_R$ is the direction towards $R$ originating at $P$ or $Q$, as is appropriate. Thus $\nu(B_R(\vv)^-) \leq \frac{1}{2}$ for each $\vv\in T_R$. In particular, the barycenter of $\nu$ is connected.

Next we show that $\Bary(\nu)$ is a segment. Let $P\in \Bary(\nu)$, and suppose that $Q$ is any other point in the barycenter. Let $\vv_Q\in T_P$ be the direction towards $Q$; then $\nu(B_P(\vv_Q)^-) \leq \frac{1}{2}$, hence $$\nu\left(\cup_{\vv\in T_P\setminus\{\vv_Q\}} B_P(\vv)^-\right) + \nu(\{P\}) \geq \frac{1}{2}\ .$$ In a similar way let $\vw_P\in T_Q$ be the direction at $Q$ towards $P$. We have  that $\nu(B_Q(\vw_{P})^-) \leq \frac{1}{2}$ and moreover \begin{align}\label{eq:baryendpts}\frac{1}{2} \leq \nu\left(\cup_{\vv\in T_P \setminus \{\vv_{Q}\}} B_P(\vv)^-\right) +\nu(\{P\}) \leq \nu(B_Q(\vw_P)^-) \leq \frac{1}{2}\ .\end{align} Hence $\frac{1}{2} = \nu(\left(\cup_{\vv\in T_P \setminus \{\vv_{Q}\}} B_P(\vv)^-\right)+\nu(\{P\})$. Thus if $\Gamma$ is any connected subgraph of $\Bary(\nu)$, it can have at most two endpoints, i.e. $\Gamma$ must be a segment. Since finite graphs exhaust $\Bary(\nu)$, it follows that if $\Bary(\nu)$ has more than one point, then it must be a segment. This segment must be closed, as it is the collection of points where $g_\nu$ is minimized. 

\end{proof}

\subsection{Asymptotics of $\MinResLoc(\phi^n)$}

We now apply the results of the previous section to the functions $g_{\mu_\phi}(x,x)$ and the sets $\MinResLoc(\phi^n)$. Note that the canonical measure $\mu_\phi$ has continuous potentials (\cite{BR} Proposition 10.7), hence the results of Proposition~\ref{prop:RL} apply.  Several of the proofs in this section will rely on the following definition:

\begin{defn}
Let $\nu$ be a probability measure on $\pberk$ that does not charge type I points. If $Q\in \Bary(\nu)$, the set of directions at $Q$ which point away from $\Bary(\nu)$ is denoted $$T_Q^*(\nu) := \{\vv\in T_Q\ : \ \textrm{ For }t\textrm{ sufficiently small, }Q+ t\vv \not\in \Bary(\nu)\}\ .$$ If $Q\not\in \Bary(\nu)$, then the directions which do not contain $\Bary(\nu)$ are similarly denoted $$T_Q^*(\nu) := \{\vv\in T_Q \ : \ \Bary(\nu)\not\subseteq B_Q(\vv)^-\}\ .$$
\end{defn}

We now give a technical lemma that is essential in proving Propostion~\ref{prop:minresepsilon}:
\begin{lemma}\label{lem:bdonmeasure}
Let $\nu$ be a probability measure with continuous potentials. For every $\epsilon >0$, there exists $\delta=\delta(\nu, \epsilon) < \frac{1}{2}$ so that for every $x\in \hberk$ with $\rho(x, \Bary(\nu))=\epsilon$ and every $\vv\in T_x$ pointing away from $\Bary(\nu)$, we have $$\nu(B_x(\vv)^-) < \delta\ .$$
\end{lemma}

\begin{proof}
In the case that $\Bary(\nu)$ is a segment $[A,B]$, for any $Q\in (A,B)$ and any direction $\vv\in T_Q^*(\nu)$ pointing away from $\Bary(\nu)$, it follows from (\ref{eq:baryendpts}) that $\nu(B_Q(\vv)^-) = 0$. Thus, it suffices to prove the assertion for the end point(s) of $\Bary(\nu)$.

Let $A$ be an endpoint of $\Bary(\nu)$, and let $x\in\hberk$ be any point with $\rho(x, A) = \epsilon$ for which $A$ is the point in $\Bary(\nu)$ that is nearest to $x$. We consider two cases:

\begin{enumerate}
\item[(i)] First, suppose that there are no directions $\vv\in T_A^*(\nu)$ with $\nu(B_A(\vv)^-) = \frac{1}{2}$. Necessarily we have 
\begin{equation}\label{eq:finitesumequilibriummass}
\sum_{\vv\in T_A^*(\nu)}\nu(B_A(\vv)^-) \leq 1\ ,
\end{equation} and our hypothesis that $\nu(B_A(\vv)^-) < \frac{1}{2}$ for all $\vv\in T_A^*(\nu)$ ensures that $$s_A=\sup_{\vv\in T_A^*(\nu)} \nu(B_A(\vv)^-) < \frac{1}{2}\ .$$ Indeed, if $s_A = \frac{1}{2}$, then there would be an infinite sequence of directions $\vv_1, \vv_2, ...\in T_A$ with $\nu(B_A(\vv_i)^-) > \frac{1}{2} - \epsilon >0$; but this would contradict (\ref{eq:finitesumequilibriummass}). Then for any $\vv\in T_A^*(\nu)$ and any $x\in B_A(\vv)^-$ with $\rho(x, A) = \epsilon$, we have $$\nu(B_x(\vw)^-) \leq \nu(B_A(\vv)^-) \leq s_A <\frac{1}{2}$$ for each $\vw\in T_x\setminus\{\vv_A\}$, where $\vv_A\in T_x$ is the direction at $x$ pointing towards $A$.

\item[(ii)] Now suppose that, for some $\vv\in T_A^*(\nu)$, we have $\nu(B_A(\vv)^-) = \frac{1}{2}$. Let $x_\epsilon$ denote a generic point in $B_A(\vv)^-$ with $\rho(x_\epsilon, A)= \epsilon$ We have that \begin{align}\label{eq:sumepsilons}\sum_{x_\epsilon}\sum_{\vv\in T_{x_\epsilon}^*(\nu)} \nu(B_{x_\epsilon}(\vv)^-) \leq \frac{1}{2}\ .\end{align} In particular, at most countably many $x_\epsilon$ have directions $\vv\in T_{x_\epsilon}^*(\nu)$ that carry mass. Let $$s_A = \sup_{x_\epsilon, \vv\in T_{x_\epsilon}^*(\nu)} \nu(B_{x_\epsilon}(\vv)^-) \leq \frac{1}{2}\ .$$ 
Note that it is impossible to have some $x_\epsilon$ and a direction $\vv\in T_{x_\epsilon}^*(\nu)$ with $\nu(B_{x_\epsilon}(\vv)^-) = \frac{1}{2}$: if this were the case, then the open annulus with boundary points $A$ and $x_\epsilon$ would carry no mass. Moreover, for every $y\in (A,x_\epsilon)$ and any $\vv\in T_y$, we would have that 
\begin{equation*}
\nu(B_y(\vv)^-) \leq \max(\nu(B_y(\vv_A)^-), \nu(B_y(\vv_{x_\epsilon})^-)= \frac{1}{2}\\
\end{equation*} so that $y\in \Bary(\nu)$ (here, $\vv_{x_\epsilon}\in T_y$ is the direction towards $x_\epsilon)$. This would contradict that $A$ is an endpoint of $\Bary(\nu)$, and so we conclude that $\nu(B_{x_\epsilon}(\vv)^-) < \frac{1}{2}$ for all $\vv\in T_{x_\epsilon}^*(\nu)$. This, together with (\ref{eq:sumepsilons}), implies that $s_A < \frac{1}{2}$, and so we have $\nu(B_{x_\epsilon}(\vv)^-) < s_A$ for every $\vv\in T_{x_\epsilon}^*(\nu)$.
\end{enumerate}

If $\Bary(\nu)$ is a single point $A$, then $s_A$ is the constant asserted in the lemma. Otherwise, if $\Bary(\nu) = [A,B]$, it suffices to take $\delta = \min(s_A, s_B)$.
\end{proof}

\begin{proof}[Proof of Proposition~\ref{prop:minresepsilon}]
Fix $0<\epsilon < 1$ and let $\delta=\delta(\mu_\phi, \epsilon /2)$ be the constant arising from Lemma~\ref{lem:bdonmeasure} with $\sigma = \mu_\phi$. Note that here we are using the constant attached to $\frac{\epsilon}{2}$ rather than the one attached to $\epsilon$. 

Observe that we can interpret the conclusion of Lemma~\ref{lem:bdonmeasure} as a statement about the slope of $g_\phi(x,x)$; namely, if $x\in \hberk$ with $\rho(x, \Bary(\mu_\phi)) = \frac{\epsilon}{2}$ and $\vv\in T_x^*(\mu_\phi)$, then $$\del_{\vv}(g_\phi)(x) = 1-2\mu_\phi(B_x(\vv)^-) > 1-2\delta>0\ .$$

Fix $R$ large enough so that $B_\rho(\Bary(\mu_\phi), \epsilon)\subseteq B_\rho(\zetaG, R)$. Choose any $y$ with $\rho(y, \Bary(\mu_\phi)) = \epsilon$, and let $x$ be the unique point on the path joining $y$ to $\Bary(\mu_\phi)$ satisfying $\rho(x, \Bary(\mu_\phi))= \frac{\epsilon}{2}$. Set $$s=(1-2\delta)\cdot\frac{\epsilon}{4}>0\ .$$ Since $g_\phi (x,x)$ is convex along segments (see Corollary~\ref{cor:gxxprops}), we have 
\begin{align*}
g_\phi(y,y) - g_\phi(x,x) & \geq \del_{\vv}(g_\phi)(x)\cdot \frac{\epsilon}{2} \\
& > (1-2\delta)\frac{\epsilon}{2} = 2s\ .
\end{align*} Equivalently, $$g_\phi(x,x) + s < g_\phi(y,y) -s\ .$$ By Theorem~\ref{thm:fnconv} above, we may choose $N$ so that for $n\geq N$, we have $$\left|\frac{1}{d^{2n}-d^n}\ord\Res_{\phi^n}(z)-g_{\phi}(z,z)\right| < s$$ for every $z\in B_{\rho}(\zetaG, R)$. In particular,
\begin{align*}
\frac{1}{d^{2n}-d^n}\ord\Res_{\phi^n}(x) &\leq g_\phi(x,x) + s\\
& < g_\phi(y,y)-s\\
& \leq \frac{1}{d^{2n}-d^n}\ord\Res_{\phi^n}(y)\ .
\end{align*}

Thus for $n\geq N$, the function $\frac{1}{d^{2n}-d^n}\ord\Res_{\phi^n}(x)$ is increasing as one moves from points at distance $\frac{\epsilon}{2}$ from $\Bary(\mu_\phi)$ to points at distance $\epsilon$ from $\Bary(\mu_\phi)$. Since $\ordRes_\phi(\cdot)$ is convex up along segments, it must attain its minimum on $B_{\rho}(\Bary(\mu_\phi), \frac{\epsilon}{2})\subseteq B_\rho(\Bary(\mu_\phi), \epsilon)$.
\end{proof}

As a consequence, we have a result that gives an interpretation of the minimal value that $g_\phi(x,x)$ takes on $\pberk$:

\begin{cor} Let $m_n =\min_{x\in\pberk}\frac{1}{d^{2n}-d^n}\ord\Res_{\phi^n}(x)$ be the value that $\frac{1}{d^{2n}-d^n} \ord\Res_{\phi^n}(x)$ takes on $\MinResLoc(\phi^n)$.  Then  $$\min_{x\in\pberk}g_\phi(x,x) = \lim_{n\to\infty}m_n\ .$$
\end{cor}
\begin{proof}
For $n\geq 1$, let $x_n\in \MinResLoc(\phi^n)$, and set $m_n = \frac{1}{d^{2n}-d^n} \ord\Res_{\phi^n}(x_n)$. Let $m_0 = \min_{x\in \pberk}g_\phi(x,x)$.

Fix $\epsilon >0$. By Proposition~\ref{prop:minresepsilon}, we may choose  $N_1$ sufficiently large so that for $n\geq N_1$, $\rho(x_n, \Bary(\mu_\phi)) < \frac{\epsilon}{2}$; by the (Lipschitz) continuity of $g_\phi(x,x)$ with respect to $\rho$ (Corollary~\ref{cor:gxxprops}), this implies \begin{align}\label{eq:mingreens} |m_0 - g_\phi(x_n, x_n)| < \frac{\epsilon}{2}\ .\end{align} Further, since $\Bary(\mu_\phi)$ is bounded, by Theorem~\ref{thm:fnconvexplicit} we may choose $N_2$ so that for $n\geq N_2$, we have \begin{align}\label{eq:ordresbound}\left|\frac{1}{d^{2n}-d^n} \ord\Res_{\phi^n}(x_n) - g_\phi(x_n,x_n) \right|  < \frac{\epsilon}{2}\ .\end{align} Taking $n\geq \max(N_1, N_2)$, Equations (\ref{eq:mingreens}) and (\ref{eq:ordresbound}) give \begin{align*}
\left| m_0 - \frac{1}{d^{2n}-d^n} \ord\Res_{\phi^n}(x_n)\right| & < \epsilon
\end{align*}

\end{proof}

We close this section by giving an asymptotic relation between $\Bary(\mu_\phi)$ and the trees $\Gamma_{\widehat{FR}, n}$:
\begin{prop}\label{prop:minrestree}
Let $\phi\in K(z)$ be a rational map of degree $d\geq 2$. Then there exists an $N=N(\phi)$ such that, for every $n\geq N$, we have $$\Bary(\mu_\phi)\subseteq \Gamma_{\widehat{FR},n}\ .$$
\end{prop}
\begin{proof}
If $\phi$ has potential good reduction, then $\Bary(\mu_\phi) = \MinResLoc(\phi^n)$ for every $n$ and there is nothing to prove. So we suppose that $\phi$ has bad reduction. In particular, $\mu_\phi$ does not charge points (see \cite{FRL}, Th\'eor\`eme E). 

We first prove the result when $\Bary(\mu_\phi)$ is a single point. Let $\Bary(\mu_\phi) = \{A\}$. Necessarily we can find two directions $\vv, \vw\in T_A$ so that $\mu_\phi(B_A(\vv)^-), \mu_\phi(B_A(\vw)^-) > 0$: if there were only one direction with $\mu_\phi(B_A(\vv)^-)>0$, then this direction would have full mass, contradicting that $A$ is in the barycenter of $\mu_\phi$ (here we are using the assumption that $\phi$ has bad reduction). Let $$\epsilon = \frac{1}{2}\min(\mu_\phi(B_A(\vv)^-),\mu_\phi(B_A(\vw)^-))\ .$$  Note that the discs $B_A(\vv)^-, B_A(\vw)^-$ have the common boundary $\del(B_A(\vv)^-) = \del(B_A(\vw)^-) = \{A\}$. Since $\mu_\phi$ does not charge points, and since the measures $\nu_{\phi^n}$ converge weakly to $\mu_\phi$, we may apply the Portmanteau theorem (\cite{BR}, Theorem A.13) to find $N$ sufficiently large so that $\nu_{\phi^n}(B_A(\vv)^-), \nu_{\phi^n}(B_A(\vw)^-) > \frac{\epsilon}{2} >0$ for $n \geq N$. In particular, there is a point of $\Gamma_{\widehat{FR}, n}$ in each of $B_A(\vv)^-, B_A(\vw)^-$ for $n\geq N$, and since $\Gamma_{\widehat{FR},n}$ is connected, it follows that $\Bary(\mu_\phi)=\{A\} \subseteq \Gamma_{\widehat{FR},n}$ whenever $n\geq N$.

A similar argument will address the case that $\Bary(\mu_\phi)$ is a segment. Let $A, B$ be the endpoints of segment, and choose $\vv\in T_A^*(\mu_\phi), \vw\in T_B^*(\mu_\phi)$ with $\mu_\phi(B_A(\vv)^-), \mu_\phi(B_B(\vw)^-)>0$. Again we let $$\epsilon = \frac{1}{2}\min(\mu_\phi(B_A(\vv)^-),\mu_\phi(B_B(\vw)^-))\ .$$ The same argument as above ensures that there is an $N$ so that, for $n\geq N$, we have $$\nu_{\phi^n}(B_A(\vv)^-), \nu_{\phi^n}(B_B(\vw)^-) >\frac{\epsilon}{2}>0\ .$$ Thus there is a point of $\Gamma_{\widehat{FR},n}$ in each of $B_A(\vv)^-, B_B(\vw)^-$ for $n\geq N$. By connectedness, it follows that $\Bary(\mu_\phi)=[A,B]\subseteq \Gamma_{\widehat{FR}, n}$.

\end{proof}

\subsection{An Example: The Failure of Hausdorff Convergence}\label{sec:barydiscussion}
Several results in this section suggest that we may be able to say something stronger than the conclusion of Proposition~\ref{prop:minresepsilon}, namely, that the sets $\MinResLoc(\phi^n)$ converge in the Hausdorff metric to $\Bary(\mu_\phi)$. However, the following example, suggested by Rob Benedetto, shows that this cannot be the case in general:

\begin{ex}\label{ex:Benedettos} Let $K=\mathbb{C}_p$ for some prime $p\neq 2$. Take $\phi(T) = \dfrac{T^2-1}{p}$. Since $\phi^{(n)}$ has even degree, $\MinResLoc(\phi^n)$ will always be a single type II point (\cite{Ru1}, Theorem 0.1). However, we can show that the barycenter of $\mu_{\phi}$ is a segment.

First we claim that $\phi$ has bad reduction. If $\phi$ had potential good reduction, then there would be a repelling fixed point $\zeta\in \hberk$, which would necessarily carry $\nu_\phi$-weight. However, we will show that the only point carrying $\nu_\phi$-weight is a non-fixed point: The classical fixed points of $\phi$ satisfy $T^2-pT-1 = 0$; by the theory of Newton polygons they both have absolute value 1, and by looking at the reduction we see that they lie off of two different directions at $\zetaG$. Thus, the tree $\Gamma_{\textrm{Fix}}$ spanned by the classical fixed points is the union of $[\gamma_1, \infty]$ and $[\gamma_2, \infty]$, where $\gamma_1, \gamma_2$ are the type I finite fixed points. These two segments meet at $\zetaG$, hence $\zetaG$ is a branch point of $\Gamma_{\textrm{Fix}}$. Moreover, because $\phi$ has constant reduction we know that $\phi(\zetaG) \neq \zetaG$, and therefore $w(\zetaG) = v_{\Gamma_{\textrm{Fix}}}(\zetaG) -2 =1$. Since $\deg(\phi) = 2$, this is the only weighted point.

We now turn to finding $\Bary(\mu_\phi)$. For this, note that the preimages of $D(0,1)$ in $K$ are the discs $D\left(1, \frac{1}{p}\right)$ and $D\left(-1, \frac{1}{p}\right)$. Arguing inductively, we claim $\phi^{-j}(D(0,1))$ is a disjoint union of $2^j$ discs $D(a,r)$ with $a\equiv \pm 1 \ \textrm{ mod }\mathfrak{m}_K$, and $r=\dfrac{1}{p^j}$. To see this, suppose that $\phi^{j-1}(D(b, s)) = D(0,1)$, and consider the preimage of $D(b,s)$. The points $w\in \phi^{-1}(b)$ satisfy \begin{align*}w^2=1+pb\ .\end{align*} Hence $\phi^{-1}(D(b,s))$ are discs centered at points $w$ with $w\equiv 1$ or $w\equiv -1$ in $\tilde{k}$. For the radius, we note that

\begin{align*}
\left|\phi(w+p^j) - \phi(w)\right| = \left|\dfrac{(w+p^j)^2-1}{p} - \dfrac{w^2-1}{p}\right| & = \left|\dfrac{2wp^j + p^{2j}}{p}\right|\\
& = \left|p^{j-1}\right|\cdot \left| 2w+p^j\right|\\
& = \dfrac{1}{p^{j-1}}\ .
\end{align*} Thus the point $w+p^j$ lies on the boundary of $\phi^{-1}(D(b, s))$, and so $\phi^{-1}(D(b, s)) = D\left(w, \dfrac{1}{p^j}\right)\sqcup D\left( w', \dfrac{1}{p^j}\right)$. Since $|w-w'| = 1$, these two discs are disjoint. Moreover, they are disjoint from any other disc in $\phi^{-j}(D(0,1))$:   suppose $D(c, t)\neq D(b, s)$ is another disc with $\phi^{j-1}(D(c,t))= D(0,1)$ whose preimages are $D(v, p^{-j}), D(v', p^{-j})$. If $D(v, p^{-j})\cap D(w, p^{-j}) \neq \emptyset$, then these discs would necessarily be equal since they have the same diameter. In particular, $D(c,t) = \phi(D(v, p^{-j})) = \phi(D(w, p^{-j})) = D(b,s)$, which contradicts that $D(c,t) \neq D(b,s)$. 

Let $\zeta_1 = \zeta_{1, p^{-1}}, \zeta_2 = \zeta_{-1, p^{-1}}$. From the above arguments, we see that $\phi^{-j}(\zetaG)$ consists of $2^n$ points, half of which lie in $\pberk \setminus B_{\vv_\infty}(\zeta_1)$ (these are the preimages $\phi^{-j}(D(0,1)) = D(a, p^{-j})$ with $a\equiv 1 \mod\mathfrak{m}_K$) and the other half of which lie in $\pberk \setminus B_{\vv_\infty}(\zeta_2)^-$. Since the measures $\frac{1}{2^n} (\phi^{n})^* \delta_{\zetaG}$ converge weakly to $\mu_\phi$ (\cite{FRL} Th\'eor\`eme A), it follows that $\mu_\phi (B_{\vv}(\zeta_1)^-) \leq \frac{1}{2}$ for every $\vv\in T_{\zeta_1} \setminus \{\vv_{\zetaG}\}$ and $\mu_{\phi}(B_{\vv}(\zeta_2)^-) \leq \frac{1}{2}$ for every $\vv\in T_{\zeta_2}\setminus \{\vv_{\zetaG}\}$. Hence the segment $[\zeta_1, \zeta_2]$ is contained in the barycenter. Moreover, the barycenter cannot be any larger.
\end{ex}
The conclusion of Proposition~\ref{prop:minresepsilon} shows that the sets $\MinResLoc(\phi^n)$ approach $\Bary(\mu_\phi)$ in some sense, though as the above example shows we do not have Hausdorff convergence. A natural question, then, is whether the sets $\MinResLoc(\phi^n)$ converge to some \emph{subset} of $\Bary(\mu_\phi)$. If this happens, the natural follow-up question is what dynamical significance this limit set has; the author does not yet have a good answer to either of these questions.

\section{Uniform Bounds on $\MinResLoc(\phi^n)$ and $\Bary(\mu_\phi)$}\label{sect:barybounds}

In this last section, we study the distance between points in $\MinResLoc(\phi^n)$ and the Gauss point, and also the distance between points in $\Bary(\mu_\phi)$ and $\zetaG$. The main lemma used in this task is the following estimate on the growth of certain coefficients of $\Phi^n$:

\begin{lemma}\label{lem:coefficientlemma}
Let $\Phi$ be a normalized lift of $\phi$. Let $\Phi^n = [F, G]$ be a normalized lift for the $n$th iterate of $\phi$, where $F(X,Y) = \alpha_D X^D + ... + \alpha_0 Y^D$, $G(X,Y) = \beta_D X^D + ... + \beta_0 Y^D$ and $D=d^n$. Then \begin{align*} \max (|\alpha_0|, |\beta_0|) &\geq |\Res(\Phi)|^{\frac{d^n-1}{d-1}}\\ \max (|\alpha_{d^n}|, |\beta_{d^n}|) & \geq |\Res(\Phi)|^{\frac{d^n-1}{d-1}}\ .\end{align*}
\end{lemma}
\begin{proof}
We observe that $|\alpha_0| = |F(0,1)|, |\alpha_D| = |F(1,0)|$ and $|\beta_0| = |G(0,1)|, |\beta_D| = |G(1,0)|$. For a pair $(x,y)$, let $||(x,y)|| = \max(|x|, |y|)$. Then by \cite{BR} Lemma 10.1, we have 
\begin{align*}
\max( |\alpha_0|, |\beta_0|) = ||\Phi^n(0,1)|| & \geq ||\Phi^{n-1}(0,1)||^d\cdot |\Res(\Phi)|\\
& \geq ||\Phi^{n-2}(0,1)||^{d^2} \cdot |\Res(\Phi)|^{1+d}\\
& \dots\\
& \geq ||(0,1)|| \cdot |\Res(\Phi)|^{1+d+...+d^{n-1}} = |\Res(\Phi)|^{\frac{d^n-1}{d-1}}\ .
\end{align*}
A similar argument holds for $\max(|\alpha_{d^n}|, |\beta_{d^n}|)$. 

\end{proof}

Lemma~\ref{lem:coefficientlemma} above gives us a bound on the size of leading and constant coefficients of the polynomials that form a normalized lift of $\phi^n$. Similar bounds appeared in the proof of \cite{Ru1} Proposition 1.8, which gave a bound for the set $\MinResLoc(\phi)$. We can use the previous lemma to strengthen this bound for iterates:

\begin{prop}\label{prop:minreslocbd}
Let $d\geq 2$ and let $R=\frac{2}{d-1} \ordRes(\phi)$. Fix a point $x\in\mathbb{P}^1(K)$. For any point $\zeta\in[\zetaG, x]$, the function $\ordRes_{\phi^n}$ satisfies \begin{equation}\label{eq:ordreslowerbound}\frac{1}{d^{2n}-d^n}\ordRes_{\phi^n}(\zeta) \geq \rho(\zetaG, \zeta)+\frac{1}{d^{2n}-d^n} \ordRes_{\phi^n}(\zetaG) - R\ .\end{equation} 

Let $\xi$ be the unique point in $[\zetaG, x]$ such that $\rho(\zetaG, \xi) = \frac{2}{d-1}\ordRes(\phi)$. Then for each $n$, the function $\ordRes_{\phi^n}(\cdot)$ is increasing along $[\xi, x]$ as one moves away from $\xi$.
\end{prop}

\begin{proof}
The proof follows \cite{Ru1} Proposition 1.8 closely. After a change of coordinates by some $\gamma\in \GL_2(\mathcal{O})$, we can assume that $x=0$. Let $\Phi^n = [F, G]$ be a normalized lift of $\phi^n$, where $D=D^n$, $F(X,Y) = a_{D}X^{D} + ... + a_0 Y^{D}$, $G(X,Y) = b_{D} X^{D} + ... + b_0 Y^{D}$, where $a_i, b_j \in \mathcal{O}$ and at least one coefficient is a unit. 

Given $A\in K^\times$, let $\tau_A(z) = Az$. In \cite{Ru1} Proposition 1.8, Rumely shows that 
\begin{align*}
\ordRes_{\phi^n}(\zeta_{0, |A|}) -& \ordRes_{\phi^n}(\zetaG)\\
&  \geq \max \left( -2D \ord(a_0) + (D^2+D)\ord(A), -2D \ord(b_0) + (D^2-D)\ord(A),\right.\\
& \left. -2D\ord(a_D) + (D-D^2)\ord(A), -2D \ord(b_D) + (-D-D^2)\ord(A)\right)\ .
\end{align*}

\noindent Using the bounds in Lemma~\ref{lem:coefficientlemma}, this gives that 
\begin{align}
\ordRes_{\phi^n}(\zeta_{0, |A|}) -& \ordRes_{\phi^n}(\zetaG)\nonumber \\
& \geq -2D\frac{d^n-1}{d-1} \ordRes(\phi) + \max\left( (D^2-D)\ord(A), (D-D^2) \ord(A)\right)\label{eq:ordresincreasing}\ .
\end{align}

Restricting ourselves to $\ord(A)>0$, the right side of (\ref{eq:ordresincreasing}) is $$-2\frac{d^{2n}-d^n}{d-1} \ordRes(\phi) + (d^{2n}-d^n) \ord(A)\ ,$$  which establishes the first claim: $$\frac{1}{d^{2n}-d^n}\ordRes_{\phi^n}(\zeta_{0, |A|}) \geq  \ord(A) - \frac{2}{d-1} \ordRes(\phi) + \frac{1}{d^{2n}-d^n} \ordRes_{\phi^n}(\zetaG)\ .$$

When $\ord(A)=0$, the left hand side of (\ref{eq:ordresincreasing}) is exactly equal to 0. Thus, if $\ord(A)$ is chosen large enough so that the right hand side of (\ref{eq:ordresincreasing}) is positive, the function $\ordRes_{\phi^n}(\cdot)$ must be increasing for all larger values of $\ord(A)$. This is attained for $$(D^2-D) \ord(A) \geq \frac{2D(d^n-1)}{d-1} \ordRes(\phi)\ ,$$ or equivalently, inserting the definition of $D=d^n$, $$\ord(A) \geq \frac{2}{d-1} \ordRes(\phi)\ .$$ 

\end{proof}

\begin{cor}\label{cor:boundonminresloc}
Let $\phi\in K(z)$ be a rational function of degree $d\geq 2$. Let $R=\frac{2}{d-1} \ordRes(\phi)$. Then for each $n$, $$\MinResLoc(\phi^n) \subseteq B_\rho(\zetaG, R)\ .$$ In particular, $\diam(\MinResLoc(\phi^n)) \leq \frac{4}{d-1} \ordRes(\phi)\ .$
\end{cor}

Note that this proposition and its corollary imply that the bound in Lemma~\ref{lem:coefficientlemma} is as sharp as one would expect in general. In particular, if the bound grew more slowly, say exponentially of order $n$ rather than order $d^n$, we could find a sequence of radii $R_n \to 0$ with $\MinResLoc(\phi^n) \subseteq B_\rho(\zetaG, R_n)$, which isn't true in general. Proposition~\ref{prop:minreslocbd} can also be used to give a lower bound for the Arakelov-Green's function:

\begin{lemma}
Let $R=\frac{2}{d-1}\ordRes(\phi)$. Fix any type I point $x$. For any point $\zeta\in [\zetaG, x]$, we have $$g_{\phi}(\zeta, \zeta) \geq \rho(\zetaG, \zeta) + g_{\phi}(\zetaG, \zetaG) - R\ .$$
\end{lemma}
\begin{proof}
We use the convergence of the functions $\frac{1}{d^{2n}-d^n}\ordRes_{\phi^n}(x)$ given in Theorem~\ref{thm:fnconv}. Let $\epsilon >0$, and fix $\zeta\in [\zetaG,x]$. We may choose $n$ large enough so that
\begin{align*}
\left|\frac{1}{d^{2n}-d^n} \ordRes_{\phi^n}(\zeta) - g_\phi(\zeta, \zeta)\right| &< \epsilon \\
\left|\frac{1}{d^{2n}-d^n} \ordRes_{\phi^n}(\zetaG) - g_\phi(\zetaG, \zetaG)\right| &< \epsilon \ .
\end{align*}

Combining this with (\ref{eq:ordreslowerbound}), we find 
\begin{align*}
g_\phi(\zeta, \zeta)+\epsilon &\geq \frac{1}{d^{2n}-d^n} \ordRes_{\phi^n}(\zeta)\\
& \geq  \rho(\zetaG, \zeta) - R + \frac{1}{d^{2n}-d^n} \ordRes_{\phi^n}(\zetaG) \\
& \geq \rho(\zetaG, \zeta) -R + g_\phi(\zetaG, \zetaG) - \epsilon\ .
\end{align*}
 Letting $\epsilon \to 0$ gives the result.

\end{proof}

We can apply this to obtain a bound on the distance of $\Bary(\mu_\phi)$ to $\zetaG$:

\begin{prop}\label{prop:RboundsonBary}
Let $R=\frac{2}{d-1} \ordRes(\phi)$ and $m_0 = \min_{x\in \pberk} g_\phi(x,x)$. Then $$\Bary(\mu_\phi) \subseteq B_\rho(\zetaG, R+m_0 - g_\phi(\zetaG))\ .$$ We further have $$\diam^\rho(\Bary(\mu_\phi)) \leq 2 (R+m_0 - g_\phi(\zetaG))\ ,$$ where $\diam^\rho$ is the diameter in the $\rho$-metric. In particular, if we choose a coordinate system so that $\zetaG\in \Bary(\mu_\phi)$, then $$\Bary(\mu_\phi) \subseteq B_\rho(\zetaG, R)$$ and $$\diam^\rho(\Bary(\mu_\phi)) \leq 2R\ .$$ 
\end{prop}

\begin{proof}
Let $R=\frac{2}{d-1}\ordRes(\phi)$, and fix $\epsilon>0$. Let $\Bary(\mu_\phi)$ be the segment $[\zeta_1, \zeta_2]$, and without loss of generality assume $\rho(\zetaG, \zeta_2) \geq \rho(\zetaG, \zeta_1)$. By the preceeding lemma, $$\rho(\zetaG, \zeta_2) \leq g_\phi(\zeta_2, \zeta_2) + R - g_\phi(\zetaG, \zetaG)\ .$$ Since $\zeta_2\in \Bary(\mu_\phi)$ and $g_\phi(x,x)$ is minimized on $\Bary(\mu_\phi)$, this gives $$\rho(\zetaG, \zeta_2) \leq R + m_0 - g_\phi(\zetaG, \zetaG)\ .$$ The last assertion follows from the fact that $g_\phi(\zetaG, \zetaG) = m_0$ if $\zetaG\in \Bary(\mu_\phi)$. The statements about the diameters are immediate.

\end{proof}

\subsection{Multipliers of Periodic Points}

Lemma~\ref{lem:coefficientlemma} can also be used to bound how repelling a type I repelling $n$-periodic point can be. More precisely, we have

\begin{prop}
Let $P$ be a type I repelling $n$-periodic point for $\phi$. Let $\Phi$ be a normalized lift for $\phi$. If $\lambda_P$ is the multiplier of $P$, we have $$|\lambda_P| \leq |\Res(\Phi)|^{-\frac{d^n-1}{d-1}}\ .$$
\end{prop}
\begin{proof}
After changing co\"ordinates by an element $\gamma\in \PGL_2(\mathcal{O})$, we may assume that $P=0$. Note that $\lambda_P$ and $|\Res(\Phi)|$ are unaffected by this type of conjugation. 

Let $D=d^n$ and $\phi^n(z) = \frac{f(z)}{g(z)}$, where $f(z) = a_D z^D+ ... + a_1 z$, $g(z) = b_D z^D+...+b_1 z+b_0$ are normalized, coprime polynomials representing the $n$th iterate of $\phi$. We have that $|a_1|\leq 1$ and $b_0 \neq 0$. The multiplier $\lambda_P$ is given $$\lambda_P = \frac{a_1}{b_0}\ .$$ By Lemma~\ref{lem:coefficientlemma}, we know $$\frac{1}{|b_0|} \leq |\Res(\Phi)|^{-\frac{d^n-1}{d-1}}\ .$$ Thus, $$|\lambda_P| = \frac{|a_1|}{|b_0|} \leq |\Res(\Phi)|^{-\frac{d^n-1}{d-1}}\ .$$
\end{proof}

\bibliography{Jacobs_Equidistribution_Resubmission_3}

\begin{thebibliography}{10}

\bibitem{BRHarmonic}
Matthew Baker and Robert Rumely.
\newblock Harmonic analysis on metrized graphs.
\newblock {\em Canadian Journal of Mathematics}, 59:225--275, 2007.

\bibitem{BR}
Matthew Baker and Robert Rumely.
\newblock {\em Potential Theory and Dynamics on the Berkovich Projective Line}.
\newblock AMS, 2010.

\bibitem{Beer}
Gerald Beer.
\newblock {\em Topologies on Closed and Closed Convex Sets}.
\newblock Kluwer Academic Publishers, 1993.

\bibitem{Ber}
V.~G. Berkovich.
\newblock {\em Spectral theory and analytic geometry over non-Archimedean
  fields}.
\newblock Amer. Math. Soc., 1990.

\bibitem{CR}
T.~Chinburg and R.~Rumely.
\newblock The capacity pairing.
\newblock {\em J. Reine Agnew. Math.}, 434:1--44, 1993.

\bibitem{DJR}
J.~Doyle, K.~Jacobs, and R.~Rumely.
\newblock Configuration of the crucial set for a quadratic rational map.
\newblock {\em Research in Number Theory}, 2, 2016.

\bibitem{XF}
Xander Faber.
\newblock {Topology and Geometry of the Berkovich Ramification Locus I}.
\newblock To appear in \emph{Manuscripta Mathematica}.

\bibitem{FRL2}
Charles Favre and Juan Rivera-Letelier.
\newblock Th\'eor\`eme d'\'e'quidistribution de brolin en dynamique p-adique.
\newblock {\em C. R. Math. Acad. Sci. Paris 339}, 4:271--276, 2004.

\bibitem{FRL}
Charles Favre and Juan Rivera-Letelier.
\newblock Th\'eorie ergodique des fractions rationelles sur un corps
  ultram\'etrique.
\newblock {\em Proc. Lond. Math. Soc.}, 1:116--154, 2010.

\bibitem{Ru1}
Robert Rumely.
\newblock The minimal resultant locus.
\newblock {\em arXiv.org:1304.1201}, April 2013.

\bibitem{Ru2}
Robert Rumely.
\newblock The geometry of the minimal resultant locus.
\newblock {\em arXiv.org:1402.6017}, Feb 2014.

\bibitem{AEC}
Joseph Silverman.
\newblock {\em The Arithmetic of Elliptic Curves}.
\newblock Springer, 1986.

\bibitem{ADS}
Joseph Silverman.
\newblock {\em The Arithmetic of Dynamical Systems}.
\newblock Springer, 2007.

\bibitem{STW}
L.~Szpiro, M.~Tepper, and P.~Williams.
\newblock Semi-stable reduction implies minimality of the resultant.
\newblock {\em Journal of Algebra}, 397.

\end{thebibliography}
\bibliographystyle{plain}

\end{document}